\documentclass[10pt,a4paper]{article}

\usepackage{amsmath,amssymb,amsthm,mathrsfs}
\usepackage{stmaryrd}  
\usepackage{graphicx}
\usepackage{xcolor}
\usepackage{soul}
\usepackage{fullpage}
\newlength{\defbaselineskip}
\newcommand{\setlinespacing}[1]%
{\setlength{\baselineskip}{#1 \defbaselineskip}}

\theoremstyle{plain}
\newtheorem{theorem}{Theorem}[section]
\newtheorem{lemma}[theorem]{Lemma}
\newtheorem{proposition}[theorem]{Proposition}

\theoremstyle{definition}

\newtheorem{ass}[theorem]{Assumption}

\newtheorem{remark}[theorem]{Remark}

\numberwithin{equation}{section}

\DeclareMathOperator*{\essinf}{ess\,inf}

\newcommand{\cF}{\mathcal{F}}

\newcommand{\bE}{\mathbb{E}}

\newcommand{\bP}{\mathbb{P}}
\newcommand{\bR}{\mathbb{R}}

\allowdisplaybreaks
\begin{document}
	
\title{Portfolio Liquidation Games with Self-Exciting Order Flow }
	\author{Guanxing Fu\thanks{The Hong Kong Polytechnic University, Department of Applied Mathematics, Hung Hom, Kowloon, Hong Kong} \quad \quad Ulrich Horst\thanks{Humboldt University Berlin, Department of Mathematics and School of Business and Economics, Unter den Linden 6, 10099 Berlin}\quad \quad Xiaonyu Xia\thanks{Humboldt University Berlin, Department of Mathematics, Unter den Linden 6, 10099 Berlin}}
	
\maketitle

\begin{abstract}
We analyze novel portfolio liquidation games with self-exciting order flow. Both the $N$-player game and the mean-field game are considered. We assume that players' trading activities have an impact on the dynamics of future market order arrivals thereby generating an additional transient price impact. Given the strategies of her competitors each player solves a mean-field control problem. We characterize open-loop Nash equilibria in both games in terms of a novel mean-field FBSDE system with unknown terminal condition. Under a weak interaction condition we prove that the FBSDE systems have unique solutions. Using a novel sufficient maximum principle that does not require convexity of the cost function we finally prove that the solution of the FBSDE systems do indeed provide existence and uniqueness of open-loop Nash equilibria.     
\end{abstract}




\section{Introduction}

	Models of optimal portfolio liquidation under market impact have received substantial consideration in the financial mathematics and the stochastic control literature in recent years. Starting with the work of Almgren and Chriss~\cite{AC-2001} existence and uniqueness of optimal liquidation strategies under various forms of market impact, trading restrictions and model uncertainty have been established by many authors including \cite{AJK-2014, bank:voss:16, FruthSchoenebornUrusov14, GatheralSchied11, GHQ-2015, GHS-2013, HXZ-2020, Kratz14, KP-2016, PZ-2018}. 
	
	One of the main characteristics of portfolio liquidation models is the terminal state constraint on the portfolio process. The constraint translates into a singular terminal condition on the associated HJB equation or an unknown terminal condition on the associated adjoint equation when applying stochastic maximum principles. In deterministic settings the state constraint is typically no challenge. In stochastic settings, however, it causes significant difficulties when proving the existence of solutions to the HJB or adjoint equation and hence in proving the existence and uniqueness of optimal trading strategies.
	
	The majority of the optimal trade execution literature allows for either instantaneous or transient impact. The first approach, initiated by Bertsimas and Lo in \cite{BL-1998} and Almgren and Chriss in \cite{AC-2001}, describes the price impact as a purely temporary effect that depends only on the present trading rate and does not influence future prices. A second approach, initiated by Obizhaeva and Wang in \cite{OW-2013}, assumes that the price impact is transient with the impact of past trades on current prices decaying over time. For single player models Graewe and Horst in \cite{GH-2017} and Horst and Xia in \cite{HX1} combined instantaneous and transient impacts into a single model. Assuming that the transient price impact follows an ordinary differential equation with random coefficients driven by the large investor's trading rate they showed that the optimal execution strategies can be characterized in terms of the solutions to multi-dimensional backward stochastic differential equations with singular terminal condition.   
	
	This paper studies a game theoretic extension of the liquidation model analyzed in \cite{GH-2017, HX1}. Our key conceptual contribution is to allow for an additional feedback of the large investors' trading activities on future market dynamics. There are many reasons why large selling orders may have an impact on future price dynamics. Extensive selling (or buying) may, for instance diminish the pool of counterparties and/or generate herding effects where other market participants start selling (or buying) in anticipation of further price decreases (or increases). Extensive selling may also attract predatory traders that employ front-running strategies. We refer Brunnermeier and Pedersen \cite{Brunnermeier-2005}, Carlin et al \cite{Carlin2007} and Schied and Sch\"oneborn \cite{SS-2009} for an in-depth analysis of predatory trading. 
	
	Specifically, we assume that the market buy and sell order dynamics follow Hawkes processes whose base intensities depend on the large investors' trading activities. Hawkes processes have recently received considerable attention in the financial mathematics literature as a powerful tool to model self-exciting order flow and its impact on stock price volatility; see \cite{BacryDelattreHoffmannMuzy2013,BacryMastromatteoMuzy2015,ElEuchRosenbaum2019a,JaissonRosenbaum2015,HXu-2020} and references therein. In the context of liquidation models, they have been employed in \cite{Alfonsi-2016,Papanicolaou-2019,Cartea-2018} albeit in very different settings. Alfonsi and Blanc in \cite{Alfonsi-2016} considered a variant of Obizhaeva-Wang model in \cite{OW-2013}, in which the continuous martingale driving the benchmark price in \cite{OW-2013} was replaced by a given point process involving mutually exciting Hawkes processes. Amaral and Papanicolaou in \cite{Papanicolaou-2019} modeled the benchmark price by the difference of two mutually exciting processes. Cartea et al in \cite{Cartea-2018} considered a liquidation model in which the investor placed limit orders whose fill rates depended on a mutually exciting ``influential'' market order flow.  In all three models the intensities of the Hawkes processes were {\sl exogenous}; in our model they are endogenously controlled by the large investors. Cay\'e and Muhle-Karbe in \cite{Caye-2016} allowed for some form of endogenous feedback of past trades on future transaction costs but did not model this using Hawkes processes. All the aforementioned papers considered single-player models while our focus is on liquidation games.
	
	We use Hawkes processes to introduce an additional transient price impact, which leads to a mean field control problem for each player. Finite player games with deterministic model parameters and transient impact were studied by Schied and coauthors in \cite{Schied-2020, Schied-2017b,Schied-2019} and by Strehle in \cite{Strehle-2018}. We allow all impact parameters and cost coefficients to be stochastic. Liquidation games with instantaneous and permanent impact and with and without strict liquidation constraint have been studied in \cite{Carlin2007,Drapeau2019,ET-2020, FH-2020, Voss-2019}. Although our mathematical framework would clearly be flexible enough to allow for an additional permanent impact we deliberately choose not to include a permanent impact as it does not alter the mathematical analysis. Instead, we choose to clarify the effects of self-exciting order flow on equilibrium liquidation strategies in a setting with only transient and instantaneous impact.  

	We consider both the finite player and the corresponding mean-field liquidation game. Mean-field games (MFGs) of optimal liquidation without strict liquidation constraint have been studied in \cite{CL-2018,C-Jai-2018,C-Jai-2018b,HJN-2015}.  Among them, Cardaliaguet and Lehalle in \cite{CL-2015} considered an MFG where each player has a different risk aversion. Casgrain and Jaimungal in \cite{C-Jai-2018,C-Jai-2018b} considered liquidation games with partial information and different beliefs, respectively. To the best of our knowledge mean-field and mean-field type games with liquidation constraint have only been analyzed by Fu et al in \cite{FGHP-2018} and Fu and Horst in \cite{FH-2020} as well as in the recent work by Evangelista and  Thamsten \cite{ET-2020}.    
	
	Our model is very different from \cite{ET-2020, FGHP-2018, FH-2020}, both conceptually and mathematically. First, with our choice of feedback effect, each player's best response function is given by the solution to a mean-field rather than a standard control problem. Second, the fact that current trading rates have an impact on future order arrivals leads to a different and much richer equilibrium dynamics. Anticipating their impact on future order arrivals, the players typically trade more aggressively initially and may alternately take long and short positions. Taking short positions in equilibrium is intuitive under our feedback mechanism. Benefiting from the inertia of market order flow the cost of over-selling may well be outweighed by the resulting additional sell order flow when closing short positions at later points in time. 
	
	It has been observed in \cite{FGHP-2018, FH-2020} before that overselling may occur in equilibrium. In these models these were pure liquidity provision effects, though, where players with initially smaller long positions and hence lesser impact cost provide liquidity to players with initially larger positions. In our model the players benefit from their impact on future order flow and oscillating trading strategies may be observed even in the single player benchmark model. At this point it is important to emphasize that anticipating one's own impact on future order flow does not generate arbitrage opportunities in single player models. In multi-player models there may indeed exist beneficial round-trips as we show by means of an explicit example.    
	
	Strongly oscillating equilibrium strategies have been observed in \cite{Schied-2017b,Schied-2019}. Unlike in their model, oscillations in our model do not arise from ``a hot potato effect'' but rather from the players' anticipated impact on future order flow. More importantly, in our model strong oscillations require extremely {\sl large} choices of impact parameters whereas in \cite{Schied-2017b,Schied-2019} they occur for small enough impact factors. We believe that it is more natural for erratic fluctuations to occur under strong rather than weak interactions. 
		   	
	We apply a stochastic method to solve the liquidation games. The stochastic maximum principle suggests that the equilibrium trading strategies in both the $N$-player game and the MFG can be characterized in terms of the solutions to coupled mean-field FBSDE systems. The forward components describe the players' optimal portfolio processes and the expected child order flow; hence their initial and in the case of the portfolio processes also terminal conditions are known. The backward components are the adjoint processes; they describe the respective equilibrium trading rates. Due to the liquidation constraint some of the terminal values are unknown. 
	
	We analyze both FBSDE systems within a common mathematical framework. Making a standard affine ansatz the system with unknown terminal condition can be replaced by an FBSDE with known initial and terminal condition, yet singular driver. Proving the existence of a small time solution to this FBSDE is not hard. The challenge is to prove the existence of a global solution on the whole time interval. Extending the continuation method for singular FBSDEs established in \cite{FGHP-2018} to our higher-dimensional system we prove that the FBSDE system does indeed have a unique solution in a certain space under a weak interaction condition that limits the impact of an individual player on the payoff of other players. Weak interaction conditions have been extensively used in the game theory literature before; see, e.g.~\cite{H-2005} and references therein. Without some form of weak interaction uniqueness of equilibria usually cannot be expected. 
	
	Subsequently, we establish a novel verification argument from which we deduce that the solution to the FBSDE system does indeed give the desired Nash equilibrium. Our maximum principle does not require convexity of the cost function as it is usually the case; see e.g. \cite[Theorem 6.4.6]{Pham}. In fact, unlike in \cite{ET-2020, FGHP-2018, FH-2020}, in our model the players' optimization problems are not convex and hence standard verification arguments do not apply. Instead, we establish a novel maximum principle that strongly relies on the liquidation constraint. Our idea is to decompose trading costs into a sum of equilibrium plus round-trip costs and to show that round-trips are costly. Finally, we prove that under an additional homogeneity assumption on the players' cost function the sequence of Nash equilibria in the $N$-player game converges in a suitable sense to the unique equilibrium in the MFG as the number of players tends to infinity. This complements the analysis in \cite{FGHP-2018} where no such convergence result was established.  
	 
	The benchmark model where all model parameters are deterministic, except the initial portfolios, is much easier to analyze. In this case, the FBSDE system reduces to an ODE system. The systems for the MFG, the single player model {and the two-player model} can be solved explicitly. The explicit solution is used to illustrate the impact of anticipating one's own impact on future order flow by {three} specific examples. 
		 
	 The remainder of this paper is organized as follows. The liquidation game is introduced in Section \ref{sec:game}. Existence and uniqueness of equilibria in both the $N$-player game and the corresponding MFG is established in Section \ref{sec:existence}. Convergence of the $N$-player equilibria to the unique MFG equilibrium is shown in Section \ref{sec-convergence}.  Numerical simulations are provided in Section \ref{sec:numerics}.   
	 
\textbf{Notation.} We use the following notation and notational conventions. 
We denote by $\langle\cdot ,\cdot \rangle$ the inner product of two vectors. For a matrix $y\in\mathbb R^{n\times m}$, denote by $|y|:=\left(\sum_{1\leq i\leq n,1\leq j\leq m}|y_{ij}|^2\right)^{1/2}$ the $2$-norm of $y$. For a $\mathbb R$-valued essentially bounded stochastic process $y$, denote by $y_{\min}$ and by $\|y\|$ its lower bound and upper bound, respectively. For a $\mathbb R^{n\times m}$-valued essentially bounded stochastic process $y$, without confusion, we still denote by $\|y\|$ its upper bound in terms of $2$-norm, i.e., $\|y\|:=(\sum_{i,j}\|y_{ij}\|^2)^{1/2}$.
%

For a filtration $\mathscr F$ we denote by $L^2_{\mathscr F}$ the space of all $\mathscr F$ progressively measurable processes such that $\|y\|_{L^2}:=\left(\mathbb E\left[\int_0^T|y_t|^2\,dt\right]\right)^{\frac{1}{2}}<\infty$. We let $\mathbb S^2_{\mathscr F}$ be the space of all $\mathscr F$ progressively measurable processes with continuous trajectories such that $\|y\|_{\mathbb S^2}:=\left(\mathbb E\left[\sup_{0\leq t\leq T}|y_t|^2\right]\right)^{\frac{1}{2}}<\infty$ and denote by $\mathcal H_{a,\mathscr F}$ the subspace of $\mathbb S^2_{\mathscr F}$ such that $\|y\|_a:=\left(\mathbb E\left[\sup_{0\leq t\leq T}\left(\frac{|y_t|}{(T-t)^a}\right)^2\right]\right)^{\frac{1}{2}}<\infty$. Finally, $L^{2,-}_{\mathscr F}$ denotes the space of all $\mathscr F$ progressively measurable processes such that for each $\epsilon>0$ it holds that $\mathbb E\left[\int_0^{T-\epsilon}|y_t|^2\,dt\right]<\infty$, and $\mathbb S^{2,-}_{\mathscr F}$ denotes the space of all $\mathscr F$ progressively measurable processes with continuous trajectories such that $\|y\|_{\mathbb S^{2,-}}:=\left(\sup_{\epsilon\geq 0}\mathbb E\left[\sup_{0\leq t\leq T-\epsilon}|y_t|^2\right]\right)^{\frac{1}{2}}<\infty$. 

Throughout, $C$ denotes a generic constant that may vary from line to line.


\section{The liquidation game}\label{sec:game}

In this paper we introduce a novel portfolio liquidation game with self-exciting order flow. Both the $N$-player game and the corresponding MFG will be considered. Our starting point is the portfolio liquidation model with instantaneous and persistent price impact analyzed in  \cite{GH-2017}. We briefly review this model in the next subsection before extending it by adding an additional feedback term of mean-field type into the dynamics of the benchmark price process. We assume throughout that randomness is described by a multi-dimensional Brownian motion $W$, unless otherwise stated, defined on a filtered probability space $(\Omega,\cF, (\cF_t),\bP)$ that satisfies the usual conditions.


\subsection{The single player benchmark model}\label{1P}


In \cite{GH-2017} the authors analyzed a liquidation model in which the investor needs to unwind an initial portfolio of $x$ shares over a finite time horizon $[0,T]$ using  absolutely continuous trading strategies. Assuming a linear-quadratic cost function, the large investor's stochastic control problem is given by   
\begin{equation} \label{control-problem1}
	\essinf_{\xi \in L^2_{\mathcal F}(0,T;\mathbb R)} \mathbb E\left[ \int_0^T\{\eta\xi_s^2+\xi_sY_s+\lambda_sX_s^2\}\,ds \right]
\end{equation}
subject to the state dynamics
\begin{equation} \label{state-dynamics1} 
\left\{\begin{aligned}
	X_t&= x - \int_0^t \xi_s\,ds,  \quad t \in [0,T],\\
	X_T &= 0, \\
	Y_t&=\int_0^t\{-\rho_sY_s+\gamma\xi_s\}\,ds, \quad t \in [0,T].
\end{aligned}\right.
\end{equation}
Here, $\eta$ and $\gamma$ are positive constants while $\rho$ and $\lambda$ are progressively measurable, non-negative and essentially bounded stochastic processes. The quantity $X_t$ denotes the number of shares the investor needs to sell at time $t \in [0,T]$, while $\xi_t$ denotes the rate at which the stock is traded at that time. The process $Y$ describes the \textsl{persistent} price impact. It can be viewed as a shift in the mid quote price caused by past trades where the impact is measured by impact factor $\gamma$. Alternatively, it can be viewed as an additional spread caused by the large investor in a block-shaped limit order book market with constant order book depth $1/\gamma>0$ as in \cite{HN-2014,OW-2013}. This results in an execution price process of the form 
\begin{equation}\label{benchmark1}
	\tilde S_t = S_t - \eta \xi_t - Y_t
\end{equation}
where $S$ is a Brownian martingale that describes the dynamics of the unaffected mid-price process. The essentially bounded process $\rho$ describes the rates at which the order book recovers from past trades. The constant $\eta > 0$ describes an additional \textit{instantaneous} impact as in \cite{AC-2001, AJK-2014, GHQ-2015,GHS-2013} among many others. The first two terms of the running cost term in \eqref{control-problem1} capture the expected liquidity cost resulting from the instantaneous and the persistent impact, respectively. The third term can be interpreted as a measure of the market risk associated with an open position. It penalizes slow liquidation.  



We are now going to introduce an additional feedback effect into the above model that accounts for the possibility of an additional order flow (``child orders'') triggered by the large investor's trading activity. To this end, we assume that the market order dynamics follows a Hawkes process with exponential kernel. Specifically, we assume that market sell and buy orders arrive according to independent counting processes $N^\pm$ with respective intensities
\[
	\zeta^\pm_t := \mu_t + \alpha \int_0^t e^{-\beta(t-s)}dN^\pm_s 
\]
where $\mu$ is a common base intensity and $\alpha,\beta$ are deterministic  coefficients that capture the impact of past orders on future order flow. In the absence of the large trader the same number of sell and buy orders arrive on average. In the presence of the large trader the base intensities change to $\mu +\xi^\pm$ where $\xi^\pm$ denotes the positive/negative part of the large investor's liquidation strategy; if $\xi_t>0$ the investor is selling, else the investor is buying. Let $\bar Z^\pm$ denote the total number of sell/buy market orders. Standard computations show that the expected number $\bar Z_t=\bE[\bar Z^+_t- \bar Z^-_t]$ of net sell order arrivals is given by
\begin{equation}
	\bar Z_t = \int_0^t \bE[\xi_s]ds + \alpha \int_0^t e^{-\beta(t-s)} \bar Z_s ds. 
\end{equation}
In particular, the expected number of (net) sell child orders due to the large investor's trading activity equals 
\begin{equation}
	C_t = \alpha \int_0^t e^{-\beta(t-s)} \bar Z_s ds. 
\end{equation}
Differentiating this equation we see that the expected (net sell) child order flow follows the dynamics 
\begin{equation}
	dC_t = \left( -(\beta-\alpha)C_t + \alpha (x_0 -  \bE X_t ) \right) dt, \quad C_0=0. 
\end{equation}
The child order flow rate increases linearly in the investor's expected traded volume $x_0 -  \bE X$. The child order flow is mean-reverting if $\frac{\alpha}{\beta} < 1$; it is well known that the Hawkes process is stable in the long term if $\frac{\alpha}{\beta}<1$.  

Starting from \eqref{state-dynamics1} but accounting for the additional child order flow in the dynamics of the market impact process $Y$ results in the following mean-field type control problem for our large investor:
\begin{equation} \label{control-problem}
	\essinf_{\xi \in L^2_{\mathcal F}(0,T;\mathbb R)} \mathbb E\left[ \int_0^T \left \{\eta_s \xi_s^2+\xi_s Y_s + \lambda_sX_s^2 \right\}\,ds	\right]
\end{equation}
subject to the state dynamics 
\begin{equation} \label{state-dynamics2} 
\left\{\begin{aligned}
	dX_t&= - \xi_t\,dt,  \quad t \in [0,T],\\
	X_0&=x_0, ~ X_T = 0, \\
	dY_t&= \left(-\rho_t Y_t +\gamma_t (\xi_t -(\beta-\alpha)C_t + \alpha (x_0 - \bE X_t) \right) dt, \quad t \in [0,T], \\
	Y_0 & = 0, \\
	dC_t &=\left( -(\beta-\alpha)C + \alpha (x_0-  \bE X) \right) dt, \quad t \in [0,T], \\
	C_0 &=0. 
\end{aligned}\right.
\end{equation}


\subsection{Many player models}\label{NP}

Let us now consider a game theoretic extension of the above liquidation model with $N$ strategically interacting investors. The trading rate, initial portfolio and portfolio process of player $i \in \{1, ..., N\}$ are denoted $\xi^i$, $\mathcal X^i$ and $X^i$, respectively. The corresponding averages are denoted $\bar \xi$, $\bar{\mathcal X}$ and $\bar X$, respectively. We assume that the initial portfolios are (not necessarily independent) square-integrable random variables.

Assuming that both the child order flow and the impact process are driven by the average trading rate results in the following mean-field type optimization problem for player $i$ given the liquidation strategies $\xi^j$ $(j \neq i)$ of all the other players:  


\begin{equation} \label{control-problem3}
	\essinf_{\xi^i \in L^2_{\mathcal F}(0,T;\mathbb R)} \mathbb E\left[ \int_0^T\eta^i_t (\xi^i_s)^2+\xi^i_s Y^i_s + \lambda^i_s(X^i_s)^2\,ds	 \right]
\end{equation}
 subject to 
\begin{equation} \label{state-dynamics3} 
\left\{\begin{aligned}
	dX^i_t&= - \xi^i_t\,ds,  \quad t \in [0,T],\\
	X^i_0&=\mathcal X^i, ~ X^i_T = 0, \\
	dY^i_t&= \left(-\rho^i_t Y_t +\gamma^i_t (\bar \xi_t -(\beta^i_t-\alpha^i_t)C^i_t + \alpha^i_t (\mathbb E[\bar{\mathcal X}]- \bE[ \bar X_t]) \right) dt, \quad t \in [0,T] \\
	Y^i_0 & = 0 \\
	dC^i_t &=\left( -(\beta^i_t-\alpha^i_t)C^i_t + \alpha^i_t (\mathbb E[\bar{\mathcal X}]-  \bE[ \bar X_t]) \right) dt, \quad t \in [0,T] \\
	C^i_0 &=0. 
\end{aligned}\right.
\end{equation}

Under the assumption that all the cost coefficients and model parameters are essentially bounded, $\cF$-progressively measurable stochastic processes and that the instantaneous impact term and the risk aversion parameters are uniformly bounded away from zero we prove that the $N$-player liquidation game admits a unique Nash equilibrium under a weak interaction condition that limits the impact of an individual player on the trading costs of other players. Since each player affects the state dynamics of other players mainly through the impact parameters $\gamma^i$ our existence and uniqueness of equilibrium result requires these parameters to be small enough and/or the unaffected processes $\eta^i$ and $\lambda^i$ to be large enough. Moreover, we require the stability condition $\frac{\alpha^i}{\beta^i} < 1$ so that child order dynamics is mean-reverting. 

\begin{remark}\label{rem:unique}
Assuming that all players trade the same stock in the same venue is natural to assume that the model parameters and cost coefficients are the same across players, except to the initial portfolios and the risk aversion parameters. We are allowing for additional heterogeneity in the players cost functions and state dynamics as this does not alter the mathematical analysis.    
\end{remark}

Under the additional assumption that the player's cost functions are homogeneous in sense that 
\begin{equation} \label{coff-hom} 
	\begin{split}
	\eta^i_t =&~ \eta\left( t, {\cal X}^i, (W^i_s)_{0 \leq s \leq t} \right), \quad
	\lambda^i_t = \lambda\left( t, {\cal X}^i, (W^i_s)_{0 \leq s \leq t} \right), \quad 
	\rho^i_t = \rho\left( t, {\cal X}^i, (W^i_s)_{0 \leq s \leq t} \right),\\
	\alpha^i_t =&~ \alpha\left( t, {\cal X}^i, (W^i_s)_{0 \leq s \leq t} \right), \quad
	\beta^i_t = \beta\left( t, {\cal X}^i, (W^i_s)_{0 \leq s \leq t} \right), \quad 
	\gamma^i_t = \gamma\left( t, {\cal X}^i, (W^i_s)_{0 \leq s \leq t} \right),
	\end{split}
\end{equation}
for independent Brownian motions $W^1, W^2, ....$ and measurable function $\eta, \lambda, \rho, \alpha, \beta, \gamma$ and
\begin{equation}\label{coff-hom2}
	{\cal X}^1, {\cal X}^2, ... \mbox{ are i.i.d.~square integrable and independent of } W^1, W^2, ... 
\end{equation}
we also prove that the equilibrium converges (in a sense to be defined) to the unique equilibrium of a corresponding MFG as the number of players tends to infinity. 
%
%
%

The MFG is obtained by first replacing the average quantities $\bar \xi$ and $\bar X$ by deterministic processes $\mu$ and $\nu$, respectively and then by solving a representative player's optimization problem subject to an additional fixed point condition. In the MFG randomness is described by a Brownian motion $\overline W$ defined on some filtered probability space $(\Omega, \overline{\cF},(\overline{\cF}_t),\bP)$ and all processes are $(\overline{\cF}_t)$-progressively measure. The corresponding MFG is then given by   
\begin{equation} \label{control-problem4}
	\essinf_{\xi \in L^2_{\overline{\mathcal F}}(0,T;\mathbb R)} \mathbb E\left[ \int_0^T\{\eta_t (\xi_t)^2+\xi_t Y_t + \lambda_t(X_t)^2\}\,dt		\right]
\end{equation}
subject to the state dynamics
\begin{equation} \label{state-dynamics4} 
\left\{\begin{aligned}
	dX_t&= - \xi_t\,ds,  \quad t \in [0,T],\\
	X_0&=\mathcal X, ~ X_T = 0, \\
	dY_t&= \left(-\rho_t Y_t +\gamma_t ( \mu_t -(\beta_t-\alpha_t)C_t + \alpha_t (\mathbb E[\mathcal X] -  \nu_t) \right) dt, \quad t \in [0,T] \\
	Y_0 & = 0 \\
	dC_t &=\left( -(\beta_t-\alpha_t)C_t + \alpha_t (\mathbb E[\mathcal X]-  \nu_t) \right) dt, \quad t \in [0,T] \\
	C_0 &=0. 
\end{aligned}\right.
\end{equation}
and the equilibrium condition 
\begin{equation} \label{state-dynamics4a} 
\left\{\begin{aligned}
	\bE[\xi^*_t(\mu,\nu)] &= \mu_t,  \quad t \in [0,T],\\
	\bE[X^*_t(\mu,\nu)] & = \nu_t,  \quad t \in [0,T].
\end{aligned}\right.
\end{equation}
Here $\xi^*(\mu,\nu)$ denotes the unique solution to \eqref{control-problem4} given $(\mu,\nu)$, and $X^*(\mu,\nu)$ is the corresponding portfolio process. 

We prove that the MFG admits a unique solution under a weak interaction condition and that the sequence of equilibria in the finite player games converges to the mean-field equilibrium if the number of players tends to infinity. 




\section{Existence and Uniqueness of Equilibria}\label{sec:existence}

In this section we provide an existence and uniqueness of equilibrium result for both the $N$-player and the mean-field liquidation games introduced in the previous section. We first characterize the equilibria of both games in terms of solutions to certain mean-field FBSDE systems with singular terminal conditions. Subsequently, we establish the existence of a unique solution to these systems within a common mathematical framework. Finally, we prove a verification argument from which we deduce the solutions to the FBSDEs do indeed provide the desired Nash equilibria.

\subsection{Characterization of open-loop equilibria}

We start by characterizing Nash equilibria in the N-player liquidation game. 
The Hamiltonian associated with the mean-field control problem \eqref{control-problem3} and \eqref{state-dynamics3} is given by 
\begin{align*}
	H^i =&~-\sum_{j=1}^N\xi^jP^{i,j} + \sum_{j=1}^NQ^{i,j}\{  -\rho^j Y^j+\gamma^j(  \bar\xi-(\beta^j-\alpha^j)C^j  )+\alpha^j\gamma^j(\mathbb E[\bar{\mathcal X}]-\mathbb E[\bar X_t] ) \} \\
	&\quad +\sum_{j=1}^NR^{i,j}\{ -(\beta^j-\alpha^j)C^j+\alpha^j(\mathbb E[ \bar{\mathcal X}]-\mathbb E[
	\bar X_t] )  \} +\xi^iY^i+\eta^i(\xi^i)^2+\lambda^i(X^i)^2.
\end{align*}
Using the same arguments as in \cite{FGHP-2018,FH-2020} the stochastic maximum principle suggests that the best response function of player $i$ given her competitors' actions is given by
\begin{equation}\label{xi-feedback}
\xi^{*,i}=\frac{P^{i,i}-Y^i-\frac{\gamma^i}{N}Q^{i,i}}{2\eta^i},
\end{equation}
where the adjoint processes $(P^{i,j},Q^{i,j},R^{i,j})$ $(j=1, ..., N)$ satisfy the stochastic system
\begin{equation} \label{system-PY}
\left\{
\begin{aligned}
-dP^{i,j}_t=&\left(2\lambda^i_tX^i_t\delta_{ij}- \frac{1}{N}\mathbb E\left[\alpha^j_t\gamma^j_tQ^{i,j}_t\right]-\frac{1}{N}\mathbb E\left[\alpha^j_tR^{i,j}_t\right] \right)\,dt-Z^{P^{i,j}}_t\,dW_t,\\
-dQ^{i,j}_t=&\left(\frac{P^{i,i}_t-Y^i_t-\frac{\gamma^i_t}{N}Q^{i,i}_t}{2\eta^i_t}\delta_{ij}  -\rho^j_t Q^{i,j}_t\right)\,dt -Z^{Q^{i,j}}_t\,dW_t,\\
-dR^{i,j}=&~\left(-\gamma^j_t(\beta^j_t-\alpha^j_t)Q^{i,j}_t-(\beta^j_t-\alpha^j_t)R^{i,j}_t\right)\,dt-Z^{R^{i,j}}_t\,dW_t\\
Q^{i,j}_T=&~R^{i,j}_T=0,
\end{aligned}\right.
\end{equation}
with a-priori unknown terminal conditions on the processes $P^{i,j}$. It can be seen from the above system that the processes $P^{i,j}$ for $j \neq i$ are not relevant for the equilibrium dynamics and that $Q^{i,j}=R^{i,j}=0$ for $j\neq i$. Putting $P^i:=P^{i,i}$, $Q^i:=Q^{i,i}$, $R^i:=R^{i,i}$ and $M^i := P^i - Y^i$ we arrive at the following coupled mean-field forward-backward system: for $i=1, ..., N$,
\begin{equation}\label{NP-fbsde-2}
\left\{\begin{aligned}
dX^i_t=&~-\frac{M^i_t-\frac{\gamma^i_t}{N}Q^i_t}{2\eta^i_t}\,dt,\\
dY^i_t=&-\rho^i_t Y^i_t+\gamma^i_t\left\{ \frac{1}{N}\sum_{j=1}^N\frac{M^j_t-\frac{\gamma^i_t}{N}Q^j_t}{2\eta^j_t} -(\beta^i_t-\alpha^i_t)C^i_t+\alpha^i_t(\mathbb E[ \bar{\mathcal X}]-\mathbb E[\bar X_t] ) \right\}\,dt,\\
dC^i_t=&~-(\beta^i_t-\alpha^i_t)C^i_t+\alpha^i_t(\mathbb E[\bar{\mathcal X}]-\mathbb E[\bar X_t])\,dt\\
-dM^i_t=&~\left(2\lambda^i_tX^i_t-\frac{1}{N}\mathbb E\left[\alpha^i_t\gamma^i_tQ^i_t\right]-\frac{1}{N}\mathbb E\left[\alpha^i_tR^i_t\right]\right)\\
&\quad -\rho^i_t Y^i_t+\gamma^i_t\left\{ \frac{1}{N}\sum_{j=1}^N\frac{M^j_t-\frac{\gamma}{N}Q^j_t}{2\eta^j_t} -(\beta^i_t-\alpha^i_t)C^i_t+\alpha^i_t(\mathbb E[ \bar{\mathcal X}]-\mathbb E[\bar X_t] ) \right\}\,dt-Z^{M^i}_t\,dW_t,\\
-dQ^{i}_t=&\left(\frac{M^{i}_t-\frac{\gamma^i_t}{N}Q^i_t}{2\eta^i_t}  -\rho^i_t Q^i_t\right)\,dt -Z^{Q^i}_t\,dW_t,\\
-dR^i_t=&~\left(-\gamma^i_t(\beta^i_t-\alpha^i_t)Q^i_t-(\beta^i_t-\alpha^i_t)R^i_t\right)\,dt-Z^{R^i}_t\,dW_t\\
X^i_0=&~\mathcal X^i, ~ Y^i_0=C^i_0=0, ~ Q^i_T=R^i_T=X^i_T=0.
\end{aligned}\right.
\end{equation}

In terms of 
\[
\underline{\mathcal S}^i=\left(\begin{matrix}
Y^i\\
C^i
\end{matrix}\right),~A^i=\left(\begin{matrix}  
\rho^i& \gamma^i(\beta^i-\alpha^i)\\
0& \beta^i-\alpha^i
\end{matrix} \right),~B^i=(B^{i,(1)},B^{i,(2)})=\left(\begin{matrix}  
\gamma^i& -\alpha^i\gamma^i\\
0& -\alpha^i
\end{matrix} \right),
\]
and
\[
\underline{\mathcal R}^i=\left(\begin{matrix}  
\frac{\alpha^i\gamma^i}{N}\sum_{j=1}^N\mathbb E[ \mathcal X^j]\\
\frac{\alpha^i}{N}\sum_{j=1}^N\mathbb E[ \mathcal X^j]
\end{matrix} \right), ~
\mathcal P^i=\left(\begin{matrix}
Q^i\\
R^i
\end{matrix}\right),~\Theta=\left(\begin{matrix}  
1\\
0
\end{matrix} \right),~\underline\chi=\frac{1}{N}\sum_{j=1}^N\left( \begin{matrix}
\xi^{*,j}\\\mathbb E[X^j]
\end{matrix} \right)
\]
the above system can be compactly rewritten as
\begin{equation}\label{NP-fbsde}
\left\{\begin{aligned}
dX^i_t=&~-\frac{M^i_t-\frac{1}{N}\left\langle B^{i,(1)}_t , \mathcal P^i_t\right\rangle}{2\eta^i_t}\,dt,\\
d\underline{\mathcal S}^i_t=&~\left(-A^i_t\underline{\mathcal S}^i_t+B^i_t\underline\chi_t+\underline{\mathcal R}^i_t\right)\,dt\\
-dM^i_t=&~\left(2\lambda^i_tX^i_t+\left\langle \Theta, -A^i_t\underline{\mathcal S}^i_t+ B^i_t\underline\chi_t+\underline{\mathcal R}^i_t\right\rangle+\frac{1}{N}\mathbb E\left[\left\langle B^{i,(2)}_t ,\mathcal P^i_t  \right\rangle\right]\right)\,dt-Z^{M^i}_t\,d W_t,\\
-d\mathcal P^{i}_t=&~\left(-(A^i_t)^\top\mathcal P^i_t+\Theta \frac{M^i_t-\frac{1}{N}\left\langle B^{i,(1)} , \mathcal P^i_t\right\rangle}{2\eta^i_t}\right)\,dt-Z^{\mathcal P^i}_t\,d W_t\\
X^i_0=&~\mathcal X^i, ~ X^i_T=0,~\underline{\mathcal S}^i_0=(0,0)^\top, ~ \mathcal P^i_T=(0,0)^\top.
\end{aligned}\right.
\end{equation}



%
%

The Hamiltonian associated with the representative player's optimization problem in the MFG reads  
\begin{equation}
\begin{split}
	H &= \eta\xi^2+\xi Y+\lambda X^2-\xi P+Q\{ -\rho Y+\gamma( \mu-(\beta-\alpha)C ) +\alpha(\mathbb E[\mathcal X]-\nu  )  \} \\
	& \quad +R\{  -(\beta-\alpha)C+\alpha( \mathbb E[\mathcal X]-\nu )  \},
\end{split}
\end{equation}
where $(P,Q,R)$ is the adjoint processes to $(X,Y,C)$. Again, the stochastic maximum principle suggests that 
the optimal strategy is given by 
\[
	\xi=\frac{P-Y}{2\eta}.
\]
Putting $M:= P-Y$ the candidate equilibrium strategy can be obtained in terms of a solution to the FBSDE system 
\begin{equation}\label{MF-fbsde-2}
\left\{\begin{split}
dX_t=&~-\frac{M_t}{2\eta_t}\,dt\\
dY_t=&~\left(-\rho_tY_t+\gamma_t\left( \mathbb E\left[ \frac{M_t}{2\eta_t} \right] -(\beta_t-\alpha_t)C_t     \right)+\alpha_t\gamma_t(\mathbb E[\mathcal X]-\mathbb E[X_t])\right)\,dt\\
dC_t=&~\left(-(\beta_t-\alpha_t)C_t+\alpha_t( \mathbb E[\mathcal X]-\mathbb E[X_t]  )\right)\,dt\\
-d M_t=&~\left(2\lambda_tX_t-\rho_tY_t+\gamma_t\left( \mathbb E\left[ \frac{M_t}{2\eta_t} \right] -(\beta_t-\alpha_t)C_t     \right)+\alpha(\mathbb E[\mathcal X]-\mathbb E[X_t])\right)\,dt-Z^{M}_t\,d\overline W_t\\
-dQ_t=&~\left(\frac{ M_t}{2\eta_t}-\rho_tQ_t\right)\,dt-Z^Q_t\,d\overline W_t\\
-dR_t=&~\left(-\gamma_t(\beta_t-\alpha_t)Q_t-(\beta_t-\alpha_t)R_t\right)\,dt-Z^R_t\,d\overline W_t\\
X_0=&~\mathcal X,~Y_0=C_0=0,~Q_T=R_T=X_T=0.
\end{split}\right.
\end{equation}

In terms of
\[
\overline{\mathcal S}=\left(\begin{matrix}
	Y\\
	C
\end{matrix}\right),~A=\left(\begin{matrix}  
\rho& \gamma(\beta-\alpha)\\
0& \beta-\alpha
\end{matrix} \right),~B=(B^{(1)},B^{(2)})=\left(\begin{matrix}  
\gamma& -\alpha\gamma\\
0& -\alpha
\end{matrix} \right)
\]\[
\overline{\mathcal R}=\left(\begin{matrix}  
\alpha\gamma\mathbb E[\mathcal X]\\
\alpha\mathbb E[\mathcal X]
\end{matrix} \right),~
{\mathcal P}=\left(\begin{matrix}
Q\\
R
\end{matrix}\right),~\overline\chi=\left( \begin{matrix}
	\mathbb E[\frac{M}{2\eta}]\\\mathbb E[X]
\end{matrix} \right),
\]
this system can be compactly rewritten as
\begin{equation}\label{MF-fbsde}
\left\{\begin{split}
dX_t=&~-\frac{M_t}{2\eta_t}\,dt\\
d\overline{\mathcal S}_t=&~\left(-A_t\overline{\mathcal S}_t+B_t\overline\chi_t+\overline{\mathcal R}_t\right)\,dt\\
-d M_t=&~\left(2\lambda_tX_t+\left\langle \Theta,-A_t\overline{\mathcal S}_t+B_t\overline\chi_t+\overline{\mathcal R}_t \right\rangle\right)\,dt-Z^M_t\,d\overline W_t\\
-d\mathcal P_t=&~ \left(-A^\top_t\mathcal P_t+\Theta\frac{M_t}{2\eta_t}  \right) \,dt-Z^{\mathcal P}_t\,d\overline W_t\\
X_0=&~\mathcal X,~X_T=0,~\overline{\mathcal S}_0=(0,0)^\top,~\mathcal P_T=(0,0)^\top.
\end{split}\right.
\end{equation}


\subsection{The mean field FBSDE} \label{MFFBSDE}

This section provides a unified approach for solving a class of linear mean-field FBSDE systems that contains the systems \eqref{NP-fbsde} and \eqref{MF-fbsde} as special cases.  Specifically, we consider the FBSDE system
\begin{equation}\label{fbsde}
\left\{\begin{aligned}
dX^i_t=&~-\frac{M^i_t-\frac{1}{N}\left\langle \widehat B^{i,(1)}_t,\mathcal P^{i}_t\right\rangle}{2\eta^i_t}\,dt,\\
d\mathcal S^i_t=&~\left(-A^i_t\mathcal S^i_t+K^i_t\chi_t+\mathcal R^i_t\right)\,dt,\\
-dM^i_t=&~\left(2\lambda^i_tX^i_t+\frac{1}{N}\mathbb E\left[\left\langle \widehat B^{i,(2)}_t ,\mathcal P^i_t  \right\rangle\right]
 +\left\langle\Theta,-A^i_t\mathcal S^i_t+K^i_t\chi_t+\mathcal R^i_t\right\rangle\right)\,dt - Z^{M^i}_t\,dW_t,\\
-d\mathcal P^i_t=&~\left(-(A^i_t)^\top \mathcal P^i_t+\Theta\frac{M^i_t-\frac{1}{N}\left\langle \widehat B^{i,(1)}_t,\mathcal P^{i}_t\right\rangle}{2\eta^i_t}\right)\,dt-Z^{\mathcal P^i}_t\,dW_t,\\
X^i_0=&~\mathcal X^i, ~  X^i_T=0,~ \mathcal S^i_0=(0,0)^\top, ~ \mathcal P^i_T=(0,0)^\top,
\end{aligned}\right.
\end{equation}
for $i=1, ..., N$ where $K^i=(K^{i,(1)},K^{i,(2)},K^{i,(3)})$ is an $\mathbb R^{2\times 3}$-valued stochastic process,   
\[
	\xi^j=\frac{M^j-\frac{1}{N}\left\langle \widehat B^{j,(1)},\mathcal P^j \right\rangle}{2\eta^j},
\]
and
\begin{equation*}
\begin{aligned}
	\chi=(\overline\xi,\mathbb E[\overline X], \mathbb E[\overline\xi] )^\top=\left( \frac{1}{N}\sum_{j=1}^N\xi^j,\mathbb E\left[\frac{1}{N}\sum_{j=1}^NX^j \right] ,\frac{1}{N}\sum_{j=1}^N\mathbb E[\xi^j]\right)^\top.
\end{aligned}
\end{equation*}

\begin{remark}\label{rmk:general-to-original}
Let $0_{2 \times 1}$ be the $2 \times 1$ zero matrix. If $K^i=(B^{i,(1)},B^{i,(2)},0_{2\times 1})$, $\widehat B^{i,(1)}=B^{i,(1)}$, $\widehat B^{i,(2)}=B^{i,(2)}$ and $\mathcal R^i=\underline{\mathcal R}^i$, then the system \eqref{fbsde} reduces to \eqref{NP-fbsde}. 
	If $N=1$, $\widehat B^{1,(1)}=\widehat B^{1,(2)}=0$, $K^{1}=(0_{2\times 1},B^{(2)},B^{(1)})$ and $\mathcal R^1=\overline{\mathcal R}$, then it reduces to \eqref{MF-fbsde}. 
\end{remark}


In order to solve the above system we make the following assumptions. 

\begin{ass}\label{Ass-1}
(i) The processes $A^i, \widehat B^i, K^i$ are progressively measurable and uniformly bounded:
\begin{equation*}
\begin{split}
\|A\| &:= \sup_i \|A^i\| < \infty, \\
\|\widehat B^{(1)}\| &:= \sup_i \|\widehat B^{i,(1)}\| <\infty,~\|\widehat B^{(2)} \| := \sup_i \| \widehat B^{i,(2)}\|<\infty, \\
\|K^{(1)}\| & := \sup_i \|K^{i,(1)}\|<\infty,~\|K^{(2)}\| := \sup_i \|K^{i,(2)}\|<\infty,~\|K^{(3)}\|  := \sup_i \|K^{i,(3)}\|<\infty.
\end{split}
\end{equation*}
(ii) There exists constants ${\widehat\rho}> 0$ and $\widetilde\rho>0$ such that for any $\mathbb R^2$-valued process $y$ and $i=1,\cdots,N,$
		\begin{equation}\label{cond-A}
	\mathbb E\left[y_t^\top A^i_t y_t\right]\geq {\widehat\rho}\mathbb E\left[y_t^\top y_t\right],\quad \mathbb E\left[y_t^\top (A^i_t)^\top y_t+y_t^\top\frac{\Theta\langle\widehat B^{i,(1)}_t,y_t\rangle}{2N\eta^i_t}\right]\geq \widetilde \rho\mathbb E\left[y_t^\top y_t\right].
	\end{equation}
(iii) The processes $\lambda^i$ and $\eta^i$ are progressively measurable, essentially bounded and there exist constants $\theta_0,\theta_1,\theta_2,\theta_3>0$ such that  $\lambda_{\min} := \inf\limits_i \lambda^i_{\min}$ and $\eta_{\min} := \inf\limits_i \eta^i_{\min}$ satisfy   
	\begin{equation}\label{ass-1}
	\left\{\begin{aligned}
	&2\lambda_{\min}-\frac{\theta_0+\theta_1+\theta_2}{2}-\left(1+\frac{1}{\theta_3}\right)\|K^{(2)}\|^2\left(\frac{\|A\|^2}{2\theta_1 {\widehat\rho^2}}+\frac{1}{2\theta_2}\right)>0,\\
	&~2\eta_{\min}-\frac{\|\widehat B^{(1)}\|}{N\widetilde\rho}-\frac{\|\widehat B^{(2)}\|^2}{2N^2\widetilde\rho^2\theta_0} -\left(1+\frac{\|\widehat B^{(1)}\|}{2N\eta_{\min}\widetilde\rho}\right)^2(1+\theta_3)(\| K^{(1)}\|+\|K^{(3)}\|)^2\left(\frac{\|A\|^2}{2\theta_1{\widehat\rho^2}}+\frac{1}{2\theta_2}\right)>0.
	\end{aligned}\right.
	\end{equation}
(iv) The random variables ${\cal X}^i$ are square integrable for each $i=1,\cdots,N$.
\end{ass}

The first assumption is standard. The second assumption essentially means that $A^i+(A^i)^\top$ is  uniformly positive definite. The third condition is similar to conditions made in \cite{FGHP-2018, FH-2020}. It states that the impact of other players on an individual player's best response function is weak enough. Specifically, it requires either the cost functions to be dominated by the terms $\eta^i_t (\xi^i_t)^2$ and $\lambda^i_t (X^i_t)^2$ that are unaffected by the choices of other players (large $\lambda_{\min}$ and large $\eta_{\min}$), or the impact of other players on an individual player's cost function {and} state dynamics to be weak enough. 

\begin{remark}
If the number of players is large enough and the processes $\alpha^i, \beta^i, \gamma^i, \rho^i$ are identical across players and constant (cf.~Remark \ref{rem:unique}), then condition \eqref{cond-A} reduces to $4\rho>\gamma^2(\beta-\alpha), \beta>\alpha$ and we can define $\widehat\rho$ by the minimum eigenvalue of the matrix $\frac{A+A^\top}{2}$, i.e.
\[
	{\widehat\rho}:=\frac{\rho+\beta-\alpha-\sqrt{(\rho+\beta-\alpha)^2-4\rho(\beta-\alpha)+\gamma^2(\beta-\alpha)^2}}{2}.  
\]
\end{remark}

We are now ready to state and prove our main result of this section. It states that that our general FBSDE system \eqref{fbsde} admits a unique solution in a suitable space if Assumption \ref{Ass-1} is satisfied. The proof is based on an extension of the continuation method introduced in \cite{FGHP-2018}.  

\begin{theorem}\label{fbsde-thm}
	Under Assumption \ref{Ass-1}, there exists a unique solution $$(X^i,\mathcal S^i,M^i,\mathcal P^i,Z^{M^i},Z^{\mathcal P^i})\in \mathcal H_{a,\mathcal F}\times \mathbb S^{2}_{\mathcal F}\times L^{2}_{\mathcal F}\times\mathcal H_{\iota,\mathcal F}\times L^{2,-}_{\mathcal F}\times L^{2}_{\mathcal F}$$
	to the FBSDE system \eqref{fbsde} for some positive constants $a<1,~\iota<1/2.$
\end{theorem}
\begin{proof}
Let $p\in[0,1], ~f^j\in L^2_{\mathcal F},~ g^j\in \mathcal H_{a,\mathcal F}$ for each $j=1,\cdots,N$, where $a$ is to be determined later. We apply the method of continuation to the following FBSDE indexed by $(p,f^j,g^j)_{j=1,\cdots,N}$:
	\begin{equation}\label{eq:fixed-point-general}
	\left\{\begin{aligned}
	d\widetilde X^i_t=&~-\frac{\widetilde M^i_t-\frac{1}{N}\left\langle \widehat B^{i,(1)}_t,\widetilde{ \mathcal P}^{i}_t\right\rangle}{2\eta^i_t}\,dt,\\
	d \widetilde {\mathcal S}^i_t=&~\left(-A^i_t\widetilde {\mathcal S}^i_t+K^i_t\widetilde {\chi_t}+\mathcal R^i_t\right)\,dt,\\
	-d\widetilde M^i_t=&~\left(2\lambda^i_t\widetilde X^i_t+\frac{1}{N}\mathbb E\left[\left\langle \widehat B^{i,(2)}_t,\widetilde {\mathcal P}^{i}_t\right\rangle\right] +\left\langle\Theta,-A^i_t\widetilde {\mathcal S}^i_t+K^i_t\widetilde{ \chi_t}+\mathcal R^i_t\right\rangle\right)\,dt-{Z}^{\widetilde M^i}_t\,dW_t,\\
	-d\widetilde {\mathcal P}^i_t=&~\left(-(A^i_t)^\top\widetilde{ \mathcal P}^i_t+\Theta\frac{p\widetilde M^i_t-\frac{1}{N}\left\langle \widehat B^{i,(1)}_t,\widetilde {\mathcal P}^{i}_t\right\rangle}{2\eta^i_t}+\Theta f^{i}_t\right)\,dt- Z^{\widetilde{\mathcal P}^i}_t\,dW_t,\\
	\widetilde X^i_0=&~x^i,~\widetilde X^i_T=0,~\widetilde {\mathcal S}^i_0=(0,0)^\top,~\widetilde {\mathcal P}^i_T=(0,0)^\top,
	\end{aligned}\right.
	\end{equation}
	where for  $j=1,\cdots,N$,
	\begin{equation*}
	\left\{\begin{aligned}
	&\widetilde\xi^j:=\frac{p\widetilde M^j-\frac{1}{N}\left\langle \widehat B^{j,(1)},\widetilde{\mathcal P}^{j}\right\rangle}{2\eta^j}+f^j\\
	&\widetilde\chi:= \frac{1}{N}\sum_{j=1}^N \left(\widetilde\xi^j,\mathbb E[p\widetilde X^j+g^j], \mathbb E[\widetilde\xi^j]\right).
%
%
	\end{aligned}\right.
	\end{equation*}

We now make the ansatz $$\widetilde M^i={\mathscr A}^i\widetilde X^i+{\mathscr B}^i.$$  Integration by parts suggests that
\begin{equation}\label{eq:mathscr-A-i}
\left\{\begin{aligned}
d{\mathscr A}^i_t&=\left(2\lambda^i_t-\frac{({\mathscr A}^i_t)^2}{2\eta^i_t}\right)\,dt-Z^{\mathscr A^i}_t\,dW^i_t,\\
\lim\limits_{t\nearrow T}{\mathscr A}^i_t&=+\infty
\end{aligned}\right.
\end{equation}
and that $\mathscr B^i$ satisfies the BSDE  
	\begin{equation}\label{eq:mathscr-B-i}
		\begin{split}
	-d{\mathscr B}^i_t=&~\left(-\frac{{\mathscr A}^i_t{\mathscr B}^i_t}{2\eta^i_t}+\frac{{\mathscr A}^i_t}{2N\eta^i_t}\left\langle \widehat B^{i,(1)}_t,\widetilde {\mathcal P}^{i}_t\right\rangle+\frac{1}{N}\mathbb E\left[\left\langle \widehat B^{i,(2)}_t,\widetilde {\mathcal P}^{i}_t\right\rangle\right]+\left\langle\Theta,-A^i_t\widetilde{ \mathcal S}^i_t+K^i_t\widetilde \chi_t+\mathcal R^i_t\right\rangle\right)\,dt \\
	& \quad -Z^{\mathscr B^i}_t\,dW_t
		\end{split}
	\end{equation}
	on $[0,T)$. It has been shown in \cite{AJK-2014,GHS-2013} that \eqref{eq:mathscr-A-i}  admits a unique solution $({\mathscr A}^i,Z^{\mathscr A^i})\in\mathcal H_{-1}\times L^2_{\mathcal F}$ and that 
\begin{equation}\label{esti-A}
\exp\left(-\int^s_r \frac{\mathscr A^i_u}{2\eta^i_u}\,du\right)\leq \left(\frac{T-s}{T-r}\right)^{b},\quad \text{ where }b:=\min_i\frac{\eta^i_{\min}}{\|\eta^i\|}\in(0,1].
\end{equation} 
The existence of a unique solution to \eqref{eq:mathscr-B-i} will be shown in Step 1 below. 

We now proceed in two steps. In Step 1 we prove that \eqref{eq:fixed-point-general} admits a unique solution when $p=0$. In Step 2 we show that once \eqref{eq:fixed-point-general} admits a unique solution for some $p \geq 0$ and for any $(f^j,g^j)_{j=1,\cdots,N}$, then the same holds if $p$ is replaced by $p+\sigma$ for every $\sigma \leq \sigma_0$ where $\sigma_0$ is a strictly positive constant that is independent of $p$. By iterating $p$ we can then solve \eqref{eq:fixed-point-general} for $p=1$. It reduces to \eqref{fbsde} by letting $f^j=g^j=0$ for all $j=1,\cdots,N$.  

\textbf{Step 1.} In this step, we prove that the system \eqref{eq:fixed-point-general} is uniquely solvable in $\mathcal H_{a,\mathcal F}\times \mathbb S^{2}_{\mathcal F}\times L^{2}_{\mathcal F}\times\mathcal H_{\iota,\mathcal F}\times L^{2,-}_{\mathcal F}\times L^2_{\mathcal F}$ for some positive constants $a<b,~\iota<1/2$ when $p=0$.

To this end, we first consider the mean-field BSDE for $(\widetilde {\mathcal P}^i,Z^{\widetilde{\mathcal P}^i})$. This BSDE has a Lipschitz continuous driver and so it has a unique solution in the space $\mathbb S^2_{\mathcal F}\times L^2_{\mathcal F}$; see e.g.~\cite[Theorem 3.1]{Buckdahn2009}. Taking conditional expectations on both sides yields
\begin{equation*}
	\begin{split}
		\widetilde {\mathcal P}^i_t=\mathbb E\left[\left.\int_t^T-(A^i_s)^\top\widetilde {\mathcal P}^i_s-\Theta\frac{\left\langle \widehat B^{i,(1)}_s, \widetilde {\mathcal P}^i_s \right\rangle}{2N\eta^i_s}+\Theta f^i_s\,ds \right|\mathcal F_t\right],
	\end{split}
\end{equation*}
which implies that
\begin{equation*}
\begin{split}
\frac{|\widetilde {\mathcal P}^i_t|}{(T-t)^\iota}\leq&~ \left(\|A\|+\frac{\|\widehat B^{(1)}\|}{2N\eta_{\min}}\right)\frac{1}{(T-t)^\iota}\mathbb E\left[\left.\int_t^T|\widetilde {\mathcal P}^i_s|\,ds\right|\mathcal F_t\right]+\frac{1}{(T-t)^\iota}\mathbb E\left[\left.\int_t^T|f^i_s|\,ds\right|\mathcal F_t\right].
\end{split}
\end{equation*}
Next, we take $\mathbb E[\sup_{0\leq t\leq T}(\cdot)^2]$ on both sides of the above inequality. By H\"older's inequality, Doob's maximal inequality and $\iota<1/2$
\begin{align*}
&~\mathbb E\left[\sup_{0\leq t\leq T}\left(	\frac{1}{(T-t)^\iota}\mathbb E\left[\left.\int_t^T|f^i_s|\,ds\right|\mathcal F_t\right]	\right)^2\right]\\
\leq&~\mathbb E\left[ \left.\sup_{0\leq t\leq T}\left( \mathbb E\left[ \int_0^T|f^i_s|^{\frac{1}{1-\iota}}\,ds\right|\mathcal F_t \right] \right)^{2(1-\iota)}    \right]\leq \left(\frac{2-2\iota}{1-2\iota}\right)^{2(1-\iota)}T^{1-2\iota}\mathbb E\left[\int_0^T|f^i_s|^2\,ds\right].
\end{align*}
Similarly, we have that
\begin{equation*}
~\mathbb E\left[\sup_{0\leq t\leq T}\left(	\frac{1}{(T-t)^\iota}\mathbb E\left[\left.\int_t^T|\widetilde {\mathcal P}^i_s|\,ds\right|\mathcal F_t\right]	\right)^2\right]\leq \left(\frac{2-2\iota}{1-2\iota}\right)^{2(1-\iota)}T^{1-2\iota}\mathbb E\left[\int_0^T|\widetilde {\mathcal P}^i_s|^2\,ds\right].
\end{equation*}	
Therefore, we conclude that
\begin{align*}
&~\mathbb E\left[\sup_{0\leq t\leq T}\left(\frac{|\widetilde {\mathcal P}^i_t|}{(T-t)^\iota}\right)^2\right]\leq~C
\left(\|\widetilde {\mathcal P}^i\|^2_{\mathbb S^2}+\|f^i\|^2_{L^2}\right),
\end{align*}
which implies that $\widetilde {\mathcal P}^i\in\mathcal H_{\iota,\mathcal F}$.
	Next, we consider the process $\widetilde {\mathcal S}^i$. Since it solves a linear ODE we get that 
	\begin{equation*}
\mathbb E\left[\sup_{0\leq t\leq T}|\widetilde {\mathcal S}^i_t|^2\right]\leq C\left( \|\mathcal R^i\|^2_{L^2}+\sum_{i=1}^N\|f^i\|^2_{L^2}+\sum_{i=1}^N\|g^i\|^2_{a}\right).
	\end{equation*}
	As a result, $\widetilde {\mathcal S}^i\in \mathbb S^2_{\mathcal F}$.
Next, we set, for $t\in [0,T)$
	\begin{equation*}
	\begin{aligned}
	{\mathscr B}^i_t:=&\mathbb E\left[\left.\int^T_te^{-\int^s_t\frac{{\mathscr A	}^i_r}{2\eta^i_r}\,dr}\left(\frac{{\mathscr A}^i_s}{2N\eta^i_s}\left\langle \widehat B^{i,(1)}_s,\widetilde {\mathcal P}^{i}_s\right\rangle+\frac{1}{N}\mathbb E\left[\left\langle \widehat B^{i,(2)}_s,\widetilde {\mathcal P}^{i}_s\right\rangle\right] \right. \right. \right. \\
	& ~ \left.  \left. \left. +\left\langle\Theta,-A^i_s\widetilde {\mathcal S}^i_s+K^i_s\widetilde \chi_t+\mathcal R^i_s\right\rangle\right) \,ds\right| \mathcal F_t \right].
	\end{aligned}
	\end{equation*}
	The estimate \eqref{esti-A} along with Doob's maximal inequality yields a constant $C > 0$ s.t.~for any $\epsilon>0,$
	\begin{equation}\label{estimate-B-general}
	\mathbb E\left[\sup_{0\leq t\leq T-\epsilon}\left| {\mathscr B}^i_t\right|^2 \right]\leq C\left(\|\widetilde {\mathcal P}^i\|^2_{\iota}+\|\widetilde{\mathcal S}^i\|^2_{\mathbb S^2}+\|\mathcal R^i\|^2_{L^2}+\sum_{i=1}^N\|f^i\|^2_{L^2}+\sum_{i=1}^N\|g^i\|^2_{a}\right).
	\end{equation}
	Thus, $ {\mathscr B}^i$ belongs to $\mathbb S^{2,-}_{\mathcal F}$ and so the martingale representation theorem yields a unique process $Z^{\mathscr B^i}\in L^{2,-}$ such that the pair $(\mathscr B^i,Z^{\mathscr B^i})$ satisfies the BSDE \eqref{eq:mathscr-B-i}.
	
	We now analyze the process $\widetilde X^i$. Taking the ansatz $\widetilde M^i={\mathscr A}^i\widetilde X^i+ {\mathscr B}^i$ into the SDE of $\widetilde X^i$ yields
	\begin{equation*}
	\widetilde X^i_t=\mathcal X^i e^{-\int^t_0\frac{{\mathscr A}^i_r}{2\eta^i_r}\,dr}-\int^t_0e^{-\int^t_s\frac{{\mathscr A}^i_r}{2\eta^i_r}\,dr}\frac{ {\mathscr B}^i_s-\frac{1}{N}\left\langle \widehat B^{i,(1)}_s,\widetilde {\mathcal P}^{i}_s\right\rangle}{2\eta^i_s}\,ds.
	\end{equation*}
Since $a<b\leq 1$, it follows from \eqref{esti-A} that
	\begin{equation*}
	\begin{aligned}
	\mathbb E\left[\sup_{0\leq t\leq T}\left|\frac{\widetilde X^i_t}{(T-t)^a}\right|^2 \right]
\leq& ~ C\left(\|\mathcal X^i\|_{L^2}+\mathbb E\left[\int^T_0\left|\frac{ {\mathscr B}^i_s}{(T-s)^a}\right|^2\,ds \right]
+\mathbb E\left[\sup_{0\leq t\leq T}\left|\frac{\widetilde {\mathcal P}^i_t}{(T-t)^\iota}\right|^2 \right]\right)\\
=&~C\left(\|\mathcal X^i\|_{L^2}+\lim_{\epsilon\rightarrow 0}\mathbb E\left[\int^{T-\epsilon}_0\left|\frac{ {\mathscr B}^i_s}{(T-s)^a}\right|^2\,ds \right]
	+\|\widetilde {\mathcal P}^i\|^2_{\iota} \right)\\
\leq&~C\left(\|\mathcal X^i\|_{L^2}+\lim_{\epsilon\rightarrow 0}\mathbb E\left[\sup_{0\leq t\leq T-\epsilon}\left| {\mathscr B}^i_t\right|^2 \right]
	+\|\widetilde {\mathcal P}^i\|^2_{\iota}\right).
	\end{aligned}
	\end{equation*}
In view of the estimate \eqref{estimate-B-general} this shows that $\widetilde X^i\in \mathcal H_{a,\mathcal F}$. 

It remains to analyze the process $\widetilde M^i$. Using the equality $\widetilde M^i={\mathscr A}^i\widetilde X^i+ {\mathscr B}^i$ and \eqref{estimate-B-general} again, we see that for each $0\leq\tau <T$
	\begin{equation}\label{M0-in-L2}
	\mathbb E\left[\sup_{0\leq t\leq \tau}\left|\widetilde  M^i_t\right|^2 \right]\leq \frac{C}{(T-\tau)^{2(1-a)}}\|\widetilde X^i\|^2_{a}+\mathbb E\left[\sup_{0\leq t\leq \tau}\left| {\mathscr B}^i_t\right|^2 \right].
	\end{equation}
 Moreover, 
	for any $\epsilon>0,$ integration by parts implies that
	\begin{equation}\label{tildeX}
	\begin{aligned}
	&\widetilde X^i_{T-\epsilon}\widetilde M^i_{T-\epsilon}-\widetilde X^i_0\widetilde M^i_0\\
	=&\int^{T-\epsilon}_0\widetilde X^i_td\widetilde M^i_t+\int^{T-\epsilon}_0\widetilde M^i_td\widetilde X^i_t\\
	=&-\int^{T-\epsilon}_0\widetilde X^i_t\left(2\lambda^i_tX^i_t+\frac{1}{N}\mathbb E\left[\left\langle \widehat B^{i,(2)}_t,\widetilde{\mathcal P}^{i}_t\right\rangle\right]+\left\langle\Theta,-A^i_t\widetilde{\mathcal S}^i_t+K^i_t \widetilde\chi_t+\mathcal R^i_t\right\rangle\right)\,dt\\
	&-\int^{T-\epsilon}_0\widetilde M^i_t\frac{\widetilde M^i_t-\frac{1}{N}\left\langle \widehat B^{i,(1)}_t,\widetilde {\mathcal P}^{i}_t\right\rangle}{2\eta^i_t}\,dt+\textrm{martingale part}.
	\end{aligned}
	\end{equation}
	Since $$\widetilde X^i_{T-\epsilon}\widetilde M^i_{T-\epsilon}={\mathscr A}^i_{T-\epsilon}(\widetilde X^i_{T-\epsilon})^2+\widetilde X^i_{T-\epsilon} {\mathscr B}^i_{T-\epsilon}\geq \widetilde X^i_{T-\epsilon} {\mathscr B}^i_{T-\epsilon},$$ by taking expectations on both sides and using \eqref{M0-in-L2} we obtain that
	\begin{equation*}
	\begin{aligned}
	&\mathbb E\left[\int^{T-\epsilon}_0 2\lambda^i_t(\widetilde X^i_t)^2+\frac{(\widetilde M^i_t)^2}{2\eta^i_t}\,dt\right]\\
	\leq& ~\epsilon' \mathbb E\left[\int^{T-\epsilon}_0 (\widetilde M^i_t)^2\,dt\right]+C(\epsilon' )\left(\|\widetilde X^i\|^2_{a}+\| {\mathscr B}^i\|^2_{\mathbb S^{2,-}}+\|\widetilde {\mathcal S}^i\|^2_{\mathbb S^{2}}+\|\widetilde {\mathcal P}^i\|^2_{\iota}+\|\mathcal R^i\|^2_{L^2}+\sum_{i=1}^N\|f^i\|^2_{L^2}+\sum_{i=1}^N\|g^i\|^2_{a}\right).
	\end{aligned}
	\end{equation*}	
	Letting $\epsilon' <\frac{1}{2\|\eta\|}$ and then taking $\epsilon\rightarrow 0,$ we conclude that $\widetilde M^i\in L^2_{\mathcal F}.$ The martingale representation theorem yields a unique $Z^{\widetilde M^i}\in L^{2,-}_{\mathcal F}$.
	
\textbf{Step 2.} We now prove that if \eqref{eq:fixed-point-general} with parameter $p$ admits a solution in $\mathcal H_{a,\mathcal F}\times\mathbb S^2_{\mathcal F}\times L^2_{\mathcal F}\times\mathcal H_{\iota,\mathcal F}\times L^{2,-}_{\mathcal F}\times L^2_{\mathcal F}$ for any $f^{i}\in L^2_{\mathcal F}, g^i\in \mathcal H_{a,\mathcal F}$, then there exists a strictly positive constant $\sigma_0$ that is independent of $p$ and $f^i,g^i$ such that the same result holds for $p+\sigma$ whenever $\sigma\in[0,\sigma_0]$.	

For any $(X^i,M^i)\in \mathcal H_{a,\mathcal F}\times L^2_{\mathcal F}$, it holds that $$f^i(M):=\sigma\frac{M^i}{2\eta^i}+f^i\in L^2_{\mathcal F}, \quad g^i(X):=\sigma X^i+g^i\in \mathcal H_{a,\mathcal F}.$$
Hence by assumption there exists a unique solution $(\widetilde X^i, \widetilde{\mathcal S}^i,\widetilde M^i,\widetilde{\mathcal P}^i,Z^{\widetilde M^i},Z^{\widetilde{\mathcal P}^i})$ in $\mathcal H_{a,\mathcal F}\times\mathbb S^{2}_{\mathcal F}\times L^{2}_{\mathcal F}\times\mathcal H_{\iota,\mathcal F}\times L^{2,-}_{\mathcal F}\times L^2_{\mathcal F}$ to the FBSDE system \eqref{eq:fixed-point-general} with $f^i=f^i(M)$ and $g^i=g^i(X)$. 
It is now sufficient to show that the mapping
	\[
	\Phi:\left((X^j)_{j=1,\cdots,N},(M^j)_{j=1,\cdots,N}\right)\mapsto\left((\widetilde X^j)_{j=1,\cdots,N},(\widetilde M^j)_{j=1,\cdots,N}\right).
	\] 
is a contraction under Assumption \ref{Ass-1}. To this end, we denote for any two stochastic processes $H$ and $H'$ their difference by $\delta H:=H-H'$ and use again the representation $\widetilde M^i=\mathscr A^i\widetilde X^i+\mathscr B^i$. 

	Integration by parts implies for any $\epsilon>0$ that \eqref{tildeX} holds with $\widetilde X^i$ replaced by $\delta \widetilde X^i$ and without non-homogenous term.
%
	Using the fact that $\delta\widetilde X^i_{T-\epsilon}\delta\widetilde M^i_{T-\epsilon}\geq \delta\widetilde X^i_{T-\epsilon}\delta {\mathscr B}^i_{T-\epsilon},$ we have that
	\begin{equation*}
	\begin{aligned}
	&\int^{T-\epsilon}_0 2\lambda^i_t(\delta\widetilde X^i_t)^2+\frac{(\delta\widetilde M^i_t)^2}{2\eta^i_t}\,dt\\
	\leq& -\delta\widetilde X^i_{T-\epsilon}\delta {\mathscr B}^i_{T-\epsilon}+\int^{T-\epsilon}_0\delta\widetilde M^i_t\frac{\left\langle \widehat B^{i,(1)}_t,\delta\widetilde {\mathcal P}^{i}_t\right\rangle}{2N\eta^i_t}\,dt+\textrm{martingale part}\\
	&-\int^{T-\epsilon}_0\delta\widetilde X^i_t\left(\frac{1}{N}\mathbb E\left[\left\langle \widehat B^{i,(2)}_t,\delta\widetilde {\mathcal P}^{i}_t\right\rangle\right]+\left\langle\Theta,-A^i_t\delta\widetilde {\mathcal S}^i_t+K^i_t\delta\widetilde \chi_t\right\rangle\right)\,dt.
	\end{aligned}
	\end{equation*}
	Taking expectations on both sides and then letting $\epsilon\rightarrow 0,$ we obtain that
	\begin{equation*}
	\begin{aligned}
	&\mathbb E\left[\int^{T}_0 2\lambda^i_t(\delta\widetilde X^i_t)^2+\frac{(\delta\widetilde M^i_t)^2}{2\eta^i_t}\,dt\right]\\
	\leq&-\mathbb E\left[\int^{T}_0\delta\widetilde X^i_t\left(\frac{1}{N}\mathbb E\left[\left\langle \widehat B^{i,(2)}_t,\delta\widetilde {\mathcal P}^{i}_t\right\rangle\right]+\left\langle\Theta,-A^i_t\delta\widetilde {\mathcal S}^i_t+K^i_t\delta\widetilde \chi_t\right\rangle\right)\,dt\right]+\mathbb E\left[\int^{T}_0\delta\widetilde M^i_t\frac{\left\langle\widehat B^{i,(1)}_t,\delta\widetilde {\mathcal P}^{i}_t\right\rangle}{2N\eta^i_t}\,dt\right].
		\end{aligned}
	\end{equation*}
Young's inequality and the inequality $|\langle x,y\rangle|\leq |x||y|$ for any two vectors $x,y$ imply that
	\begin{equation}\label{inequality-general}
	\begin{aligned}
	&\mathbb E\left[\int^{T}_0 2\lambda^i_t(\delta\widetilde X^i_t)^2+\frac{(\delta\widetilde M^i_t)^2}{2\eta^i_t}\,dt\right]\\
	\leq&\frac{\theta_0}{2}\mathbb E\left[\int^{T}_0(\delta\widetilde X^i_t)^2\,dt\right]+\frac{\|\widehat B^{(2)}\|^2}{{2}N^2\theta_0}\mathbb E\left[\int^{T}_0|\delta\widetilde {\mathcal P}^i_t|^2\,dt\right]\\
	&+\frac{\theta_1}{2}\mathbb E\left[\int^{T}_0(\delta\widetilde X^i_t)^2\,dt\right]+\frac{\|A\|^2}{2\theta_1}\mathbb E\left[\int^{T}_0|\delta\widetilde {\mathcal S}^i_t|^2\,dt\right]\\
	&+\frac{\theta_2}{2}\mathbb E\left[\int^{T}_0(\delta\widetilde X^i_t)^2\,dt\right]+\frac{1}{2\theta_2}\mathbb E\left[\int^{T}_0\left|K^i_t\delta\widetilde  \chi_t\right|^2\,dt\right]\\
	&+\frac{\theta}{2}\mathbb E\left[\int^{T}_0\left(\frac{\delta\widetilde M^i_t}{2\eta^i_t}\right)^2\,dt\right]+\frac{\|\widehat B^{(1)}\|^2}{{2}N^2\theta}\mathbb E\left[\int^{T}_0|\delta\widetilde {\mathcal P}^{i}_t|^2\,dt\right].
	\end{aligned}
	\end{equation}
	Applying It\^o's formula for $|\delta\widetilde P^i_t|^2$, we have that
	\begin{equation*}
	\begin{aligned}
	-|\delta\widetilde {\mathcal P}^i_t|^2=&-2\int_t^T (\delta\widetilde {\mathcal P}^i_s)^\top\left(-(A^i_t)^\top\delta\widetilde {\mathcal P}^i_t+\Theta\frac{p\delta\widetilde M^i_t+\sigma \delta M^i_t-\frac{1}{N}\left\langle \widehat B^{i,(1)}_t,\delta\widetilde {\mathcal P}^{i}_t\right\rangle}{2\eta^i_t}\right)\,ds\\
	&+\int_t^T |\delta Z^{\widetilde{\mathcal P}^i}_s|^2\,ds+2\int_t^T (\delta\widetilde {\mathcal P}^i_s)^\top\delta Z^{\widetilde{\mathcal P}^i}_s\,dW_s.
	\end{aligned}
	\end{equation*}
	Recalling the condition \eqref{cond-A} and using Young's inequality $\langle x,y\rangle \leq \frac{\widetilde\rho}{2}|x|^2+\frac{1}{2\widetilde\rho}|y|^2$, we obtain that
	\begin{equation}\label{ineq-Pi-general}
	\begin{aligned}
	\mathbb E\left[\int^T_0|\delta\widetilde {\mathcal P}^i_t|^2\,dt\right]&\leq\frac{1}{\widetilde{\rho}^2}\mathbb E\left[\int^T_0\left(\frac{p\delta\widetilde M^i_t+\sigma \delta M^i_t}{2\eta^i_t}\right)^2\,dt\right].
		\end{aligned}
	\end{equation}
Using similar arguments on $|\delta\widetilde  {\mathcal S}^i_t|^2$, we get that
	\begin{equation*}
	\begin{aligned}
	|\delta\widetilde  {\mathcal S}^i_t|^2= 2\int_0^t  (\delta\widetilde  {\mathcal S}^i_s)^\top\left(-A^i_t\delta\widetilde  {\mathcal S}^i_t+K^i_t\delta\widetilde  \chi_t\right)\,ds,
	\end{aligned}
	\end{equation*}
	and 
	\begin{equation}\label{estimate:S-step2}
	\begin{aligned}
	\mathbb E\left[\int^T_0|\delta\widetilde {\mathcal S}^i_t|^2\,dt\right]
		\leq
	\frac{1}{\widehat\rho^2}\mathbb E\left[\int^T_0\left|K^i_t\delta\widetilde \chi_t\right|^2\,dt\right].
	\end{aligned}
	\end{equation}
Recalling the definition of $\widetilde\chi$ and $\widetilde\xi^j$, Remark \ref{rmk:general-to-original}, and using Young's inequality again, we have that
	\begin{align*}
&\mathbb E\left[\int^T_0\left|K^i_t\delta\widetilde \chi_t\right|^2\,dt\right]\\
\leq& (1+\theta_3)(\|K^{(1)}\|+\|K^{(3)}\|)^2\frac{1}{N}\sum_{j=1}^N\mathbb E\left[\int_0^T(\delta\widetilde\xi^j_t)^2\,dt\right] +\left(1+\frac{1}{\theta_3}\right)\|K^{(2)}\|^2\frac{1}{N}\sum_{j=1}^N\mathbb E\left[\int_0^T(p\delta\widetilde X^j_t+\sigma\delta X^j_t)^2\,dt\right]  \\
\leq&
(1+\theta_3)(\|K^{(1)}\|+\|K^{(3)}\|)^2\frac{1}{N}\sum_{j=1}^N\mathbb E\left[\int^T_0 (1+\varepsilon)\left(\frac{p\delta\widetilde M^j_t+\sigma \delta M^j_t}{2\eta^j_t}\right)^2+\left(1+\frac{1}{\varepsilon}\right)\left(\frac{\frac{1}{N}\left\langle \widehat B^{j,(1)}_t,\delta\widetilde {\mathcal P}^{j}_t\right\rangle}{2\eta^j_t}\right)^2\,dt\right]\\
&+\left(1+\frac{1}{\theta_3}\right)\|K^{(2)}\|^2\frac{1}{N}\sum_{j=1}^N\mathbb E\left[\int^T_0\left(p\delta\widetilde X^j_t+\sigma\delta X^j_t\right)^2\,dt\right].
	\end{align*}

Letting $\varepsilon:=\frac{\|\widehat B^{(1)}\|}{2N\eta_{\min}\widetilde\rho}$, from the above estimate and \eqref{ineq-Pi-general} we have that
	\begin{equation}\label{esti:K-chi}
	\begin{split}
	&\mathbb E\left[\int^T_0\left|K_t\delta\widetilde \chi_t\right|^2\,dt\right]\\
	\leq&
	(1+\theta_3)(\|K^{(1)}\|+\|K^{(3)}\|)^2\frac{1}{N}\sum_{j=1}^N\mathbb E\left[\int^T_0 \left(1+\varepsilon+\left(1+\frac{1}{\varepsilon}\right)\frac{\|\widehat B^{(1)}\|^2}{4N^2\eta_{\min}^2\widetilde\rho^2}\right)\left(\frac{p\delta\widetilde M^j_t+\sigma \delta M^j_t}{2\eta^j_t}\right)^2\,dt\right]\\
	&+\left(1+\frac{1}{\theta_3}\right)\|K^{(2)}\|^2\frac{1}{N}\sum_{j=1}^N\mathbb E\left[\int^T_0\left(p\delta\widetilde X^j_t+\sigma\delta X^j_t\right)^2\,dt\right]\\
	=&
	(1+\theta_3)(\|K{(1)}\|+\|K^{(3)}\|)^2\frac{1}{N}\sum_{j=1}^N\mathbb E\left[\int^T_0 \left(1+\frac{\|\widehat B^{(1)}\|}{2N\eta_{\min}\widetilde\rho}\right)^2\left(\frac{p\delta\widetilde M^j_t+\sigma \delta M^j_t}{2\eta^j_t}\right)^2\,dt\right]\\
	&+\left(1+\frac{1}{\theta_3}\right)\|K^{(2)}\|^2\frac{1}{N}\sum_{j=1}^N\mathbb E\left[\int^T_0\left(p\delta\widetilde X^j_t+\sigma\delta X^j_t\right)^2\,dt\right].
	\end{split}
\end{equation}

	Recalling the inequality \eqref{inequality-general}, collecting the estimates \eqref{ineq-Pi-general}-\eqref{esti:K-chi} and taking sum from 1 to $N$ on both sides we get
	\begin{align*}
	&\left(2\lambda_{\min}-\frac{\theta_0+\theta_1+\theta_2}{2}\right)\sum_{i=1}^N\mathbb E\left[\int^{T}_0 (\delta\widetilde X^i_t)^2\,dt\right]+\left(2\eta_{\min}-\frac{\theta}{2}\right)\sum_{i=1}^N\mathbb E\left[\int^{T}_0\left(\frac{\delta\widetilde M^i_t}{2\eta^i_t}\right)^2\,dt\right]\\
	\leq&\left[\frac{1}{N^2\widetilde{\rho}^2}\left(\frac{\|\widehat B^{(2)}\|^2}{2\theta_0}+\frac{\|\widehat B^{(1)}\|^2}{2\theta}\right)\right.\\
&\left.+(1+\theta_3)(\|K^{(1)}\|+\|K^{(3)}\|)^2\left(\frac{\|A\|^2}{2\theta_1\widehat\rho^2}+\frac{1}{2\theta_2}\right)\left(1+\frac{\|\widehat B^{(1)}\|}{2N\eta_{\min}\widetilde\rho}\right)^2\right]\sum_{i=1}^N\mathbb E\left[\int^T_0\left(\frac{p\delta\widetilde M^i_t+\sigma \delta M^i_t}{2\eta^i_t}\right)^2\,dt\right]\\
&+\left(1+\frac{1}{\theta_3}\right)\|K^{(2)}\|^2\left(\frac{\|A\|^2}{2\theta_1\widehat\rho^2}+\frac{1}{2\theta_2}\right)\sum_{i=1}^N\mathbb E\left[\int^T_0\left(p\delta\widetilde X^i_t+\sigma \delta X^i_t\right)^2\,dt\right]\\
	\leq&(1+\varepsilon)\left[\frac{1}{N^2\widetilde{\rho}^2}\left(\frac{\|\widehat B^{(2)}\|^2}{2\theta_0}+\frac{\|\widehat B^{(1)}\|^2}{2\theta}\right)\right.\\
	&\left.+(1+\theta_3)(\|K^{(1)}\|+\|K^{(3)}\|)^2\left(\frac{\|A\|^2}{2\theta_1\widehat\rho^2}+\frac{1}{2\theta_2}\right)\left(1+\frac{\|\widehat B^{(1)}\|}{2N\eta_{\min}\widetilde\rho}\right)^2\right]\sum_{i=1}^N\mathbb E\left[\int^T_0\left(\frac{\delta\widetilde M^i_t}{2\eta^i_t}\right)^2\,dt\right]\\
&+(1+\varepsilon)\left(1+\frac{1}{\theta_3}\right)\|K^{(2)}\|^2\left(\frac{\|A\|^2}{2\theta_1\widehat\rho^2}+\frac{1}{2\theta_2}\right)\sum_{i=1}^N\mathbb E\left[\int^T_0\left(\delta\widetilde X^i_t\right)^2\,dt\right]\\
&+C\left(1+\frac{1}{\varepsilon}\right)\sigma\left(\sum_{i=1}^N\mathbb E\left[\int^T_0\left(\delta M^i_t\right)^2\,dt\right]+\sum_{i=1}^N\mathbb E\left[\int^T_0\left(\delta X^i_t\right)^2\,dt\right]\right).
	\end{align*}

	Thus, choosing $\theta=\frac{\|\widehat B^{(1)}\|}{N\widetilde\rho}$ and choosing $\varepsilon$ small enough, the assumption \eqref{ass-1} yields
	\begin{equation*}
	\begin{aligned}
	&\sum_{i=1}^N\mathbb E\left[\int^{T}_0\left(\delta\widetilde M^i_t\right)^2\,dt\right]+\sum_{i=1}^N\mathbb E\left[\int^T_0\left(\delta\widetilde X^i\right)^2\,dt\right]\\
	&\leq C\sigma\left(\sum_{i=1}^N\mathbb E\left[\int^T_0\left(\delta M^i_t\right)^2\,dt\right]+\sum_{i=1}^N\mathbb E\left[\int^T_0\left(\delta X^i_t\right)^2\,dt\right]\right).
\end{aligned}
\end{equation*}
Furthermore, going back to the dynamics of $\widetilde X^i$ and using $\widetilde M^i=\mathscr A^i\widetilde X^i+\mathscr B^i$, we have that
	\begin{equation*}
\begin{aligned}
\sum_{i=1}^N\mathbb E\left[\sup_{0\leq t\leq T}\left|\frac{\delta\widetilde X^i}{(T-t)^a}\right|^2\right]\leq C\sigma\left(\sum_{i=1}^N\mathbb E\left[\int^T_0\left(\delta M^i_t\right)^2\,dt\right]+\sum_{i=1}^N\mathbb E\left[\int^T_0\left(\delta X^i_t\right)^2\,dt\right]\right).
\end{aligned}
\end{equation*}
Hence, when $\sigma$ is small enough, the mapping $\Phi$ is a contraction. Iterating $p$ finitely many times until $p=1$ and letting $f^i=g^i=0$, we obtain the desired result.  
\end{proof}

%


\subsection{Verification}\label{sec:verification}

Having established the existence of a unique solution to the respective FBSDEs, the candidate optimal strategies are well defined. In this section we provide a verification result that shows that the candidate strategy \eqref{xi-feedback} does indeed define a Nash equilibrium of the $N$-player game \eqref{control-problem3}-\eqref{state-dynamics3}. Our analysis is based on a novel sufficient stochastic maximum principle that does not require convexity of the cost function as it is usually the case; see e.g. \cite[Theorem 6.4.6]{Pham}. Instead, our argument strongly relies on the liquidation constraint $X^i_T=0$. The following is the main result of this section. 

\begin{theorem}\label{thm:verification}
Let $(X^i,\mathcal S^i,M^i,\mathcal P^i,Z^{M^i},Z^{\mathcal P^i})\in \mathcal H_{a,\mathcal F}\times S^{2}_{\mathcal F}\times L^{2}_{\mathcal F}\times\mathcal H_{\iota,\mathcal F}\times L^{2,-}_{\mathcal F}\times L^{2}_{\mathcal F}$ $(i=1,\cdots, N)$ be the unique solution to the FBSDE system \eqref{NP-fbsde}. Under Assumption \ref{Ass-1} with the inequalities in (iii) replaced by the following slightly stronger condition
\begin{equation}\label{stronger-ass}
	\left\{\begin{split}
	&~\lambda_{\min}-\frac{\theta_1+\theta_2}{2}-\frac{\|B^{(2)}\|^2}{N^2}\left(1+\frac{1}{\theta_3} \right)\left(\frac{\|A\|^2}{2\theta_1\widehat\rho^2}+\frac{1}{2\theta_2}\right)\geq 0,\\
	&~ \eta_{\min}-(1+\theta_3)\frac{\|B^{(1)}\|^2}{N^2} \left(\frac{\|A\|^2}{2\theta_1\widehat\rho^2}+\frac{1}{2\theta_2}\right)\geq 0,
	\end{split}\right.
\end{equation}
the processes $\xi^{*}=(\xi^{*,1},\cdots,\xi^{*,N})$ forms a unique open-loop Nash equilibrium of the $N$-player game \eqref{control-problem3}-\eqref{state-dynamics3}, where
\begin{equation*}
\xi^{*,i}=\frac{M^{i}-\frac{1}{N}\left\langle B^{i,(1)},\mathcal P^{i}\right\rangle}{2\eta^i}.
\end{equation*}
\end{theorem}



\begin{remark}
\begin{itemize}
	\item[(1)] If Assumption \ref{Ass-1} (iii) holds, then the condition \eqref{stronger-ass} holds for all $N\geq 2$.
	\item[(2)] The impact process $Y$ is exogenous in the optimization problem of the MFG. Thus, the convexity requirement for the standard sufficient maximum principle holds. We omit the proof of the verification result, which is standard.
\end{itemize}
\end{remark}

In what follows we denote by $(X^i,\underline{\mathcal S}^i)$ the states corresponding to the strategy profile $(\xi^i,(\xi^{*,j})_{j\neq i})$ and by $(X^{*,i},\underline{\mathcal S}^{*,i})$ the states corresponding to the strategy profile $(\xi^{*,i},(\xi^{*,j})_{j\neq i})$. Moreover, we put 
\[
	\underline\chi:=\frac{1}{N}(\xi^i,\mathbb E[X^i] )^\top+\frac{1}{N}\sum_{j\neq i}(\xi^{*,j},\mathbb E[X^{*,j}])^\top  \quad \mbox{and} \quad 
	\underline\chi^*:=\frac{1}{N}\sum_{j=1}^N(\xi^{*,j},\mathbb E[X^{*,j}])^\top.
\]
Then it holds that
	\begin{equation*}
		\left\{\begin{split}
			dX^i_t=&~-\xi^i_t\,dt\\
			d\underline{\mathcal S}^i_t=&~-A^i_t\underline{\mathcal S}^i_t+B^i_t\underline{\chi}_t+\underline{\mathcal R}^i_t\,dt  
		\end{split}\right.
		\quad \mbox{and} \quad
		\left\{\begin{split}
		dX^{*,i}_t=&~-\xi^{*,i}_t\,dt\\
		d\underline{\mathcal S}^{*,i}_t=&~-A^i_t\underline{\mathcal S}^{*,i}_t+B^i_t\underline{\chi}^*_t+\underline{\mathcal R}^i_t\,dt.
		\end{split}\right.
	\end{equation*}

The admissibility of the candidate $\xi^*$ has already been established; in particular, $X^{*,i}_T=0$ for each $i=1,\ldots,N$ because $X^{*,i}\in\mathcal H_{\alpha,\mathcal F}$. It remains to prove that 
\begin{equation*}
	J^i(\xi^{*,i},\xi^{*,-i})\leq J^i(\xi^i,\xi^{*,-i})
\end{equation*}
for each $1\leq i\leq N$ and any admissible control $\xi^i$. To this end, we prove that the cost $J^i(\xi^{i},\xi^{*,-i})$ can be decomposed into the equilibrium cost plus  the cost of a round-trip strategy as  
\begin{equation} \label{cost-decomposition}
\begin{split}
	J^i(\xi^i,\xi^{*,-i}) & = J^i(\xi^{*,i},\xi^{*,-i}) + \mathbb E\left[\int_0^T \eta^i_t\left(\xi^i_t-\xi^{*,i}_t\right)^2+\lambda^i_t\left(X^i_t-X^{*,i}_t\right)^2 \right.\\
	& \qquad \qquad \qquad \qquad \qquad  \left. +\left(X^i_t-X^{*,i}_t\right) \left\langle\Theta,-A^i_t(\underline{\mathcal S}^i_t-\underline{\mathcal S}^{*,i}_t)+B^i_t(\underline\chi_t-\underline\chi^*_t)\right\rangle\,dt\right]
\end{split}
\end{equation}
and that the additional cost is non-negative under Assumption \eqref{stronger-ass}. As a byproduct of our verification result we thus obtain that round-trips are costly in equilibrium. 

In order to prove the decomposition \eqref{cost-decomposition} we proceed in various steps. In a first step, we establish an alternative representation of the cost function. 

\begin{lemma}\label{lem:rewritten-cost}
The cost associated with the strategy $(\xi^i,\xi^{*,-i})$ can be rewritten as
\begin{equation*}
\begin{aligned}
J^i(\xi^i,\xi^{*,-i})
=&\mathbb E\left[\int_0^T X^i_t\left\langle \Theta,-A^i_t\underline{\mathcal S}^i_t+B^i_t\underline\chi_t+\underline{\mathcal R}^i_t\right\rangle+\eta^i_t(\xi^i_t)^2+\lambda^i_t(X^i_t)^2\,dt\right].
\end{aligned}
\end{equation*}
\end{lemma}
\begin{proof}
Using integration by parts and $X^i_T=0, \underline{\mathcal S}^i_0=0$, we have that
\begin{equation}\label{integration-by-part-1}
\begin{aligned}
0
=\mathbb E\left[ X^i_T\left\langle \Theta,\underline{\mathcal S}^i_T\right\rangle-X^i_0\left\langle \Theta,\underline{\mathcal S}^i_0\right\rangle\right]
=\mathbb E\left[\int^T_0 X^i_t\left\langle \Theta,-A^i_t\underline{\mathcal S}^i_t+B^i_t\underline\chi_t+\underline{\mathcal R}^i_t\right\rangle-\xi^i_t\left\langle \Theta,\underline{\mathcal S}^i_t\right\rangle\,dt\right].
\end{aligned}
\end{equation}
As a result, 
\begin{equation*}
\begin{aligned}
J^i(\xi^i,\xi^{*,-i})
=&\mathbb E\left[\int_0^T\xi^i_t\left\langle\Theta,\underline{\mathcal S}^i_t\right\rangle+\eta^i_t(\xi^i_t)^2+\lambda^i_t(X^i_t)^2\,dt\right]\\
=&\mathbb E\left[\int_0^T X^i_t\left\langle \Theta,-A^i_t\underline{\mathcal S}^i_t+B^i_t\underline\chi_t+\underline{\mathcal R}^i_t\right\rangle+\eta^i_t(\xi^i_t)^2+\lambda^i_t(X^i_t)^2\,dt\right].
\end{aligned}
\end{equation*}
\end{proof}

In view of Lemma \ref{lem:rewritten-cost}, it holds
\begin{equation}\label{def-i}
\begin{split}
	 &~ J^i(\xi^i,\xi^{*,-i})-J^i(\xi^{*,i},\xi^{*,-i}) \\
	~~ = & ~ \mathbb E\left[\int_0^T X^i_t\left\langle \Theta,-A^i_t\underline{\mathcal S}^i_t+B^i_t\underline\chi_t+\underline{\mathcal R}^i_t\right\rangle+		\eta^i_t(\xi^i_t)^2+\lambda^i_t(X^i_t)^2\,dt\right]\\
&~-\mathbb E\left[\int_0^T X^{*,i}_t\left\langle \Theta,-A^i_t\underline{\mathcal S}^{*,i}_t+B^i_t\underline\chi^*_t+\underline{\mathcal R}^i_t\right\rangle +\eta^i_t(\xi^{*,i}_t)^2+\lambda^i_t(X^{*,i}_t)^2\,dt \right]\\
 =:&~ \mathbb{I}.
\end{split}
\end{equation}
It remains to bring the term on the right-hand side in equation \eqref{def-i} into the form \eqref{cost-decomposition}. For thus, let 
\begin{equation}\label{eq:II}
\begin{split}
	\mathbb{II}:=&\mathbb E\left[\int_0^T\left(2\eta^i_t\xi^{*,i}_t+\frac{1}{N}\left\langle B^{i,(1)}_t,\mathcal P^{i}_t\right\rangle\right)\left(\xi^i_t-\xi^{*,i}_t\right)\,dt\right]\\
	& +\mathbb E \left[ \int_0^T \left(X^i_t-X^{*,i}_t\right) \left(2\lambda^i_tX^{*,i}_t+\frac{1}{N}\mathbb E\left[\left\langle B^{i,(2)}_t ,\mathcal P^i_t  \right\rangle\right]+\left\langle\Theta,-A^i_t\underline{\mathcal S}^{*,i}_t+B^i_t\underline\chi^*_t+\underline{\mathcal R}^i_t\right\rangle\right)\, dt \right].
\end{split}
\end{equation}
Heuristically, this term equals $\mathbb E \left[ \int_0^T \left((X^i_t - X^{*,i}_t )dM^i_t  - M^i_t (dX^{i}_t - dX^{*,i}_t) \right) \right]$. In view of the liquidation constraint, using an integration by parts argument, we expect that $\mathbb{II}=0$ in which case it remains to bring the difference $\mathbb{I}-\mathbb{II}$ into the form \eqref{cost-decomposition}. 

\begin{lemma}
	The representation \eqref{cost-decomposition} holds true.  
\end{lemma}
\begin{proof}
We proceed in two steps. In a first step, we prove that $\mathbb{II}=0$. Indeed, integration by parts on $[0,T-\epsilon]$ yields that
\begin{equation*}
\begin{split}
&\mathbb E\left[M^i_{T-\epsilon}\left(X^i_{T-\epsilon} - X^{*,i}_{T-\epsilon}\right)-M^i_{0}\left(X^i_{0}-X^{*,i}_{0}\right)\right]\\
=&-\mathbb E\left[\int^{T-\epsilon}_0 M^i_t\left(\xi^i_t-\xi^{*,i}_t\right)\,dt\right]\\
& -\mathbb E\left[\int^{T-\epsilon}_0 \left(X^i_t-X^{*,i}_t\right) \left(2\lambda^i_tX^{*,i}_t+\frac{1}{N}\mathbb E\left[\left\langle B^{i,(2)}_t,\mathcal P^i_t\right\rangle\right] +\left\langle\Theta,-A^i_t\underline{\mathcal S}^{*,i}_t+B^i_t\underline\chi^*_t+\underline{\mathcal R}^i_t\right\rangle\right)\,dt\right].
\end{split}
\end{equation*}
Letting $\epsilon \rightarrow 0$, a similar argument as in the proof of \cite[Proposition 2.14]{FGHP-2018} yields that
$$\lim\limits_{\epsilon\rightarrow 0}\mathbb E\left[M^i_{T-\epsilon}(X^i_{T-\epsilon}-X^{*,i}_{T-\epsilon})\right]=0.$$
Thus, dominated convergence implies
\begin{equation*}\label{mx-int}
\begin{aligned}
&-\mathbb E\left[\int^{T}_0 M^i_t\left(\xi^i_t-\xi^{*,i}_t\right)\,dt\right]\\
=&\mathbb E\left[\int^{T}_0 \left(X^i_t-X^{*,i}_t\right) \left(2\lambda^i_tX^{*,i}_t+\frac{1}{N}\mathbb E\left[\left\langle  B^{i,(2)}_t,\mathcal P^i_t\right\rangle\right] +\left\langle\Theta,-A^i_t\underline{\mathcal S}^{*,i}_t+B^i_t\underline\chi^*_t+\underline{\mathcal R}^i_t\right\rangle\right)\,dt\right].
\end{aligned}
\end{equation*}
Putting the preceding equation into \eqref{eq:II} implies that
\begin{equation*}
\begin{aligned}
	\mathbb{II}=&\mathbb E\left[\int_0^T\left(2\eta^i_t\xi^{*,i}_t+\frac{1}{N}\left\langle B^{i,(1)}_t,\mathcal P^{i}_t\right\rangle-M^i_t\right)\left(\xi^i_t-\xi^{*,i}_t\right)\,dt\right]=0.
\end{aligned}
\end{equation*}
Using integration by parts again yields that 
\begin{equation*}
\begin{aligned}
0=&\mathbb E\left[\left\langle \mathcal P^i_T,\underline{\mathcal S}^i_T-\underline{\mathcal S}^{*,i}_T\right\rangle-\left\langle \mathcal P^i_0,\underline{\mathcal S}^i_0-\underline{\mathcal S}^{*,i}_0\right\rangle\right]\\
=&\mathbb E\left[\int^T_0\left\langle (A^i_t)^\top \mathcal P^i_t-\Theta\xi^{*,i}_t,\underline{\mathcal S}^i_t-\underline{\mathcal S}^{*,i}_t\right\rangle+\left\langle \mathcal P^i_t,-A^i_t(\underline{\mathcal S}^i_t-\underline{\mathcal S}^{*,i}_t)+B^i_t(\underline\chi_t-\underline\chi^*_t)\right\rangle\,dt\right]\\
=&\mathbb E\left[\int^T_0-\xi^{*,i}_t\left\langle\Theta,\underline{\mathcal S}^i_t-\underline{\mathcal S}^{*,i}_t\right\rangle+\left\langle \mathcal P^i_t,B^i_t(\underline\chi_t-\underline\chi^*_t)\right\rangle\,dt\right]\\
=&\mathbb E\left[\int^T_0 -X^{*,i}_t\left\langle\Theta,-A^i_t(\underline{\mathcal S}^i_t-\underline{\mathcal S}^{*,i}_t)+B^i_t(\underline\chi_t-\underline\chi^*_t)\right\rangle\right]+\mathbb E\left[X^{*,i}_T\left\langle\Theta,\underline{\mathcal S}^i_T-\underline{\mathcal S}^{*,i}_T\right\rangle-X^{*,i}_0\left\langle\Theta,\underline{\mathcal S}^i_0-\underline{\mathcal S}^{*,i}_0\right\rangle\right]\\
&\qquad +\mathbb E\left[\int^T_0\left\langle \mathcal P^i_t,B^i_t(\underline\chi_t-\underline\chi^*_t)\right\rangle\,dt\right]\\
=&\mathbb E\left[\int^T_0 -X^{*,i}_t\left\langle\Theta,-A^i_t(\underline{\mathcal S}^i_t-\underline{\mathcal S}^{*,i}_t)+B^i_t(\underline\chi_t-\underline\chi^*_t)\right\rangle\right]+\mathbb E\left[\int^T_0\left\langle \mathcal P^i_t,B^i_t(\underline\chi_t-\underline\chi^*_t)\right\rangle\,dt\right],
\end{aligned}
\end{equation*}
where in the fourth equality we use the liquidation constraint $X_T^{*,i}=0$.
Using that $\mathbb E[\mathbb E[x]y]=\mathbb E[x]\mathbb E[y]=\mathbb E[x\mathbb E[y]]$ for any random variables $x$ and $y$, the second term in the above sum can be rewritten as 
\begin{equation*}
\begin{aligned}
&\mathbb E\left[\int^T_0\left\langle \mathcal P^i_t,B^i_t(\underline\chi_t-\underline\chi^*_t)\right\rangle\,dt\right]\\
=&\frac{1}{N}\mathbb E\left[\int^T_0\left\langle \mathcal P^i_t,B^{i,(1)}_t\right\rangle\left(\xi^i_t-\xi^{*,i}_t\right)\,dt\right]+\frac{1}{N}\mathbb E\left[\int^T_0\left\langle \mathcal P^i_t,B^{i,(2)}_t\right\rangle\mathbb E\left[X^i_t-X^{*,i}_t\right]\,dt\right]\\
=&\frac{1}{N}\mathbb E\left[\int^T_0\left\langle \mathcal P^i_t,B^{i,(1)}_t\right\rangle\left(\xi^i_t-\xi^{*,i}_t\right)\,dt\right]+\frac{1}{N}\mathbb E\left[\int^T_0\mathbb E\left[\left\langle \mathcal P^i_t,B^{i,(2)}_t\right\rangle\right]\left(X^i_t-X^{*,i}_t\right)\,dt\right].
\end{aligned}
\end{equation*}
Thus, 
\begin{equation}\label{ps-int}
\begin{aligned}
&\mathbb E\left[\int^T_0 X^{*,i}_t\left\langle\Theta,-A^i_t(\underline{\mathcal S}^i_t-\underline{\mathcal S}^{*,i}_t)+B^i_t(\underline\chi_t-\underline\chi^*_t)\right\rangle\right]\\
=&\frac{1}{N}\mathbb E\left[\int^T_0\left\langle \mathcal P^i_t,B^{i,(1)}_t\right\rangle\left(\xi^i_t-\xi^{*,i}_t\right)\,dt\right]+\frac{1}{N}\mathbb E\left[\int^T_0\mathbb E\left[\left\langle \mathcal P^i_t,B^{i,(2)}_t\right\rangle\right]\left(X^i_t-X^{*,i}_t\right)\,dt\right].
\end{aligned}
\end{equation}
Note that
\begin{equation}\label{diff-i}
\begin{aligned}
\mathbb{I}-\mathbb{II}=&\mathbb E\left[\int_0^T \eta^i_t\left(\xi^i_t-\xi^{*,i}_t\right)^2+\lambda^i_t\left(X^i_t-X^{*,i}_t\right)^2+X^i_t\left\langle\Theta,-A^i_t(\underline{\mathcal S}^i_t-\underline{\mathcal S}^{*,i}_t)+B^i_t(\underline\chi_t-\underline\chi^*_t)\right\rangle\right.\\
&\left.-\frac{1}{N}\left\langle B^{i,(1)}_t,\mathcal P^{i}_t\right\rangle\left(\xi^i_t-\xi^{*,i}_t\right)-\frac{1}{N}\mathbb E\left[\left\langle \mathcal P^i_t,B^{i,(2)}_t\right\rangle\right]\left(X^i_t-X^{*,i}_t\right)\,dt\right].
\end{aligned}
\end{equation}
Plugging \eqref{ps-int} into \eqref{diff-i}, we get the desired representation.
\end{proof}

We are now ready to finish the proof of the verification result. 

\textsc{Proof of Theorem \ref{thm:verification}.}
Using the constants appearing in \eqref{inequality-general}, we have
\begin{equation*}
	\begin{split}
		&~\mathbb E\left[\int_0^T\left(X^i_t-X^{*,i}_t\right)\left\langle\Theta,-A^i_t(\underline{\mathcal S}^i_t-\underline{\mathcal S}^{*,i}_t)+B^i_t(\underline\chi_t-\underline\chi^*_t)\right\rangle\,dt\right]\\
		\leq&~\frac{\theta_1+\theta_2}{2}\mathbb E\left[\int_0^T | X^i_t-X^{*,i}_t |^2  \right]+\frac{\|A\|^2}{2\theta_1}\mathbb E\left[\int_0^T| \underline{\mathcal S}^i_t-\underline{\mathcal S}^{*,i}_t |^2\,dt  \right]+\frac{1}{2\theta_2}\mathbb E\left[\int_0^T|B^i_t(\underline\chi_t-\underline\chi^*_t)|^2\,dt \right].
	\end{split}
\end{equation*}
The dynamics $\underline{\mathcal S}^i_t-\underline{\mathcal S}^{*,i}_t=\int^t_0\left( -A^i_s(\underline {\mathcal S}^i_s-\underline{\mathcal S}^{*,i}_s)+B^i_s(\underline\chi_s-\underline\chi^*_s)\right)\,ds$  and the estimate leading to \eqref{estimate:S-step2} imply
\[
	\mathbb E\left[\int_0^T| \underline{\mathcal S}^i_t-\underline{\mathcal S}^{*,i}_t |^2\,dt  \right]\leq \frac{1}{\widehat\rho^2}\mathbb E\left[\int_0^T|B^i_t(\underline \chi_t-\underline\chi^*_t )|^2\,dt  \right].
\]
Thus, 
\begin{equation*}
\begin{split}
&~\mathbb E\left[\int_0^T\left(X^i_t-X^{*,i}_t\right)\left\langle\Theta,-A^i_t(\underline{\mathcal S}^i_t-\underline{\mathcal S}^{*,i}_t)+B^i_t(\underline\chi_t-\underline\chi^*_t)\right\rangle\,dt\right]\\
\leq&~\frac{\theta_1+\theta_2}{2}\mathbb E\left[\int_0^T | X^i_t-X^{*,i}_t |^2  \right]+\left(\frac{\|A\|^2}{2\theta_1\widehat\rho^2}+\frac{1}{2\theta_2}\right)\mathbb E\left[\int_0^T|B^i_t(\underline\chi_t-\underline\chi^*_t)|^2\,dt \right]\\
\leq&~\left(\frac{\theta_1+\theta_2}{2}+\frac{\|B^{(2)}\|^2}{N^2}\left(1+\frac{1}{\theta_3} \right)\left(\frac{\|A\|^2}{2\theta_1\widehat\rho^2}+\frac{1}{2\theta_2}\right)\right)\mathbb E\left[\int_0^T | X^i_t-X^{*,i}_t |^2\,dt  \right]\\
&\quad +(1+\theta_3)\frac{\|B^{(1)}\|^2}{N^2}\left(\frac{\|A\|^2}{2\theta_1\widehat\rho^2}+\frac{1}{2\theta_2}\right)\mathbb E\left[ \int_0^T| \xi^i_t-\xi^{*,i}_t  |^2\,dt \right].
\end{split}
\end{equation*}

Due to the decomposition \eqref{cost-decomposition} and Assumption \eqref{stronger-ass}, we have by 
\begin{equation*}
\begin{aligned}
&~J(\xi,\xi^{*,-i})-J(\xi^{*,i},\xi^{*,-i})\\
\geq&~\left(\lambda_{\min}-\frac{\theta_1+\theta_2}{2}-\frac{\|B^{(2)}\|^2}{N^2}\left(1+\frac{1}{\theta_3} \right)\left(\frac{\|A\|^2}{2\theta_1\widehat\rho^2}+\frac{1}{2\theta_2}\right)\right)\mathbb E\left[\int^T_0\left(X^i_t-X^{*,i}_t\right)^2\,dt\right]\\
&~ +\left(\eta_{\min}-(1+\theta_3)\frac{\|B^{(1)}\|^2}{N^2} \left(\frac{\|A\|^2}{2\theta_1\widehat\rho^2}+\frac{1}{2\theta_2}\right)\right)\mathbb E\left[\int^T_0\left(\xi^i_t-\xi^{*,i}_t\right)^2\,dt\right]\\
\geq&~ 0.
\end{aligned}
\end{equation*}
\hfill $\Box$

\subsection{Approximation by penalization}\label{sec:penalization}

It has been shown in various settings that the optimal trading strategies in models in which open positions are increasingly penalized converge to optimal trading strategies in models where full liquidation is required; see, e.g.~\cite{ET-2020, FGHP-2018, HX1} for details. If the strict liquidation constraint is replaced by a penalization $n(X^i_T)^2$ of open positions at the terminal time, the  FBSDE system \eqref{fbsde} changes to  

\begin{equation}\label{penalized-FBSDE}
\left\{\begin{aligned}
dX^i_t=&~-\frac{M^i_t-\frac{1}{N}\langle \widehat B^{i,(1)}_t,\mathcal P^{i}_t\rangle}{2\eta^i_t}\,dt,\\
d\mathcal S^i_t=&~\left(-A^i_t\mathcal S^i_t+K^i_t\chi_t+\mathcal R^i_t\right)\,dt,\\
-dM^i_t=&~\left(2\lambda^i_tX^i_t+\frac{1}{N}\mathbb E\left[\left\langle \widehat B^{i,(2)}_t,\mathcal P^{i}_t\right\rangle\right]
+\left\langle\Theta,-A^i_t\mathcal S^i_t+K^i_t\chi_t+\mathcal R^i_t\right\rangle\right)\,dt - Z^{M^i}_t\,dW_t,\\
-d\mathcal P^i_t=&~\left(-(A^i_t)^\top \mathcal P^i_t+\Theta\frac{M^i_t-\frac{1}{N}\langle \widehat B^{i,(1)}_t,\mathcal P^{i}_t\rangle}{2\eta^i_t}\right)\,dt-Z^{\mathcal P^i}_t\,dW_t,\\
X^i_0=&~\mathcal X^i,~ \mathcal S^i_0=(0,0)^\top, ~M^i_T=2nX^i_T-\mathcal S^{i,(1)}_T,~ \mathcal P^i_T=(0,0)^\top,
\end{aligned}\right.
\end{equation}
where $\mathcal S^{i,(1)}$ is the first component of $\mathcal S^i$. The same arguments as in the proof of \cite[Lemma 4.5]{FGHP-2018} show that 
\[
	\mathbb E\left[\int_0^T |\mathcal P^{i,n}_t-\mathcal P^i_t|^2  \,dt  \right]+	\mathbb E\left[\int_0^T |M^{i,n}_t-M^i_t|^2  \,dt  \right]+	\mathbb E\left[\int_0^T |\mathcal S^{i,n}_t-\mathcal S^i_t|^2  \,dt  \right]\rightarrow 0.
\]	

From this, we immediately obtain that the model with liquidation constraint can be approximated by a sequence of model with increasing penalization. Specifically, using the same arguments as in the proof of \cite[Theorem 4]{FGHP-2018} it is not difficult to prove the following approximation result. 

\begin{proposition}
Let $(X^i,\mathcal S^i,M^i,\mathcal P^i,Z^{M^i},Z^{\mathcal P^i})$ and $(X^{i,n},\mathcal S^{i,n},M^{i,n},\mathcal P^{i,n},Z^{M^{i,n}},Z^{\mathcal P^{i,n}})$ be the solutions of \eqref{fbsde} and \eqref{penalized-FBSDE}, respectively. Then, 
	\[
		\mathbb E\left[\sup_{0\leq t\leq T}|X^{i,n}_t-X^i_t|^2\right]+\mathbb E\left[\sup_{0\leq t\leq T}|\mathcal S^{i,n}_t-\mathcal S^i_t|^2\right]\rightarrow 0\quad\textrm{as }n\rightarrow\infty.
	\]
\end{proposition}

\section{From many player games to mean-field games}\label{sec-convergence}

In this section we prove the convergence of the Nash equilibria in the $N$-player game to the Nash equilibrium of the corresponding MFG under the homogeneity condition \eqref{coff-hom}. This is achieved by establishing the convergence of the solutions to the FBSDE system \eqref{NP-fbsde} to the solution to the corresponding mean-field FBSDE \eqref{MF-fbsde} as $N \to \infty$. More precisely, let 
\[
	\left( \overline X^i,\overline {\mathcal S}^i,\overline M^i,\overline {\mathcal P}^i, Z^{\overline M^i}, Z^{\overline{\mathcal P}^i} \right) \in \mathcal H_{a,\mathcal F}\times \mathbb S^{2}_{\mathcal F}\times L^{2}_{\mathcal F}\times\mathcal H_{\iota,\mathcal F}\times L^{2,-}_{\mathcal F}\times L^{2}_{\mathcal F}
\]
be the unique solution to the mean-field FBSDE \eqref{MF-fbsde} with $W = W^i$, $\mathcal X = \mathcal X^i$, $\lambda=\lambda^i$, $\eta=\eta^i$, $\rho=\rho^i$, $\alpha=\alpha^i$, $\beta=\beta^i$ and $\gamma=\gamma^i$.   
Using the Yamada-Watanabe result for mean-field FBSDE established in \cite[Lemma 3.2]{FGHP-2018}, there exists a measurable function $\Sigma$ independent of $i$ such that
\[
	(\overline X^i_t,\overline{\mathcal S}^i_t,\overline M^i_t,\overline{\mathcal P}^i_t)=\Sigma(t,\mathcal X^i,W^i_{\cdot \wedge t}).
\]
In particular, the mean field equilibrium state and control satisfy
\[
	\nu_t=\mathbb E[\overline X^i_t]\qquad\textrm{and}\qquad\mu_t =\mathbb E\left[\frac{\overline M^i_t}{2\eta^i_t}\right].
\] 

\begin{lemma}
It hold that
\begin{equation}\label{c:mu-converge}
	\mathbb E\left[\int^T_0\left(\frac{1}{N}\sum_{j=1}^N\frac{\overline M^j_t}{2\eta^j_t}-\mu_t\right)^2\,dt\right]\xrightarrow{N\rightarrow \infty} 0,
\end{equation}
and
\begin{equation}\label{c:x-converge}
\mathbb E\left[\sup_{0\leq t\leq T}\left(\frac{1}{N}\sum_{j=1}^N \overline X^{j}_t-\nu_t\right)^2\,dt\right]\xrightarrow{N\rightarrow \infty} 0.
\end{equation} 
\end{lemma}
\begin{proof}
By Theorem \ref{fbsde-thm}, there exists a constant $C$ independent of $i$ such that
\begin{equation}\label{xi-est}
\mathbb E\left[\int^T_0\left(\frac{\overline M^i_t}{2\eta^i_t}\right)^2\,dt\right]\leq C.
\end{equation}
Since $\frac{\overline M^k_t}{2\eta^k_t}$ and $\frac{\overline M^j_t}{2\eta^j_t}$ are independent and identically distributed for $k\neq j$ it follows that
\begin{equation*}
\begin{aligned}
	& \mathbb E\left[\int^T_0\left(\frac{1}{N}\sum_{j=1}^N\frac{\overline M^j_t}{2\eta^j_t}-\mu_t\right)^2\,dt\right] \\
	=&\frac{1}{N^2}\mathbb E\left[\int^T_0\sum_{k\neq j}\left(\frac{\overline M^k_t}{2\eta^k_t}-\mu_t\right)\left(\frac{\overline M^j_t}{2\eta^j_t}-\mu_t\right)\,dt\right]
	+	\frac{1}{N^2}\mathbb E\left[\int^T_0\sum_{k=1}^N\left(\frac{\overline M^k_t}{2\eta^k_t}-\mu_t\right)^2\,dt\right] \\
	\leq & \frac{4C}{N} \xrightarrow{N\rightarrow \infty} 0.
\end{aligned}
\end{equation*}
By considering the dynamics of $\overline X^i$, the convergence \eqref{c:x-converge} follows.
\end{proof}

Let $(X^i,\mathcal{\underline S}^i,M^i,\mathcal P^i,Z^{M^i},Z^{\mathcal P^i})$ be the unique solution of \eqref{NP-fbsde} and
%
\[
	\left( \delta X^i,\delta \mathcal S^i,\delta M^i,\delta\mathcal P^i,\delta Z^{M^i},\delta Z^{\mathcal P^i} \right)
	:= \left( X^i-\overline X^i,\underline{\mathcal S}^i-\overline {\mathcal S}^i,M^i-\overline M^i,\mathcal P^i-\overline {\mathcal P}^i,Z^{M^i}- Z^{\overline{M}^i}e_i,Z^{\mathcal P^i}- Z^{\overline{\mathcal P}^i}e_i \right),
\]	
where $e_i$ denotes the $i$th unit vector in $\bR^N$. The FBSDE 
\begin{equation}\label{eq:diff-convergence}
\left\{\begin{aligned}
d\delta X^i_t=&~-\frac{\delta M^i_t-\frac{1}{N}\left\langle B^{i,(1)}_t,\mathcal P^{i}_t\right\rangle}{2\eta^i_t}\,dt,\\
d\delta \mathcal S^i_t=&~\left(-A^i_t\delta \mathcal S^i_t+B^i_t\delta \chi_t+\delta\mathcal R^i_t\right)\,dt,\\
-d\delta M^i_t=&~\left(2\lambda^i_t\delta X^i_t+\frac{1}{N}\mathbb E\left[\left\langle B^{i,(2)}_t ,\mathcal P^i_t  \right\rangle\right]+\left\langle\Theta,-A^i_t\delta \mathcal S^i_t+B^i_t\delta \chi_t+\delta\mathcal R^i_t\right\rangle\right)\,dt-\delta Z^{M^i}_t\,dW^{}_t,\\
-d\delta \mathcal P^i_t=&~\left(-(A^i_t)^\top\delta \mathcal P^i_t+\Theta\frac{\delta M^i_t-\frac{1}{N}\left\langle B^{i,(1)}_t, \mathcal P^{i}_t\right\rangle}{2\eta^i_t}\right)\,dt-\delta Z^{\mathcal P^i}_t\,dW^{}_t,\\
\delta X^i_0=&~0, ~\delta X^i_T=0,~\delta \mathcal S^i_0=(0,0)^\top,~ \delta \mathcal P^i_T=(0,0)^\top,
\end{aligned}\right.
\end{equation}
where
\[
	\delta\chi=\left( \frac{1}{N}\sum_{j=1}^N\frac{M^j-\frac{1}{N}\left\langle B^{j,(1)}, \mathcal P^j \right\rangle}{2\eta^j}-\mathbb E\left[ \frac{\overline M^i}{2\eta^i} \right],\frac{1}{N}\sum_{j=1}^N\mathbb E[X^j]-\mathbb E[\overline X^i] \right)^\top
\]
and
\[
  \delta\mathcal R^i=\left(\frac{\alpha^i\gamma^i}{N}\sum_{j=1}^N\mathbb E[\mathcal X^j]-\alpha^i\gamma^i\mathbb E[\mathcal X^i],\frac{\alpha^i}{N} \sum_{j=1}^N\mathbb E[\mathcal X^j]-\alpha^i\mathbb E[\mathcal X^i]   \right)^\top\equiv (0,0)^\top.
\]
has a unique solution. This allows us to establish the convergence of the Nash equilibria on the $N$-player game to the mean field solution as $N \to \infty$.

\begin{theorem}
{Let \eqref{cond-A} and \eqref{ass-1} hold for all $N$ large enough.}	The following convergence holds
\begin{equation*}
	\mathbb E\left[\int_0^T|\delta M^i_t|^2\,dt\right] + \mathbb E\left[\sup_{0\leq t\leq T}|\delta X^i_t|^2\,dt\right] \xrightarrow{N\rightarrow\infty}0.
\end{equation*}
As a result, the optimal strategy of player $i$ in the $N$-player game converges to the one in MFG, i.e.,
\[
	\mathbb E\left[ \int_0^T| \xi^{*,i,N}_t-\overline{\xi}^{*,i}_t |^2\,dt \right]\rightarrow 0.
\]
where $\xi^{*,i,N}:=\frac{M^i-\frac{1}{N}\left\langle B^{i,(1)}, \mathcal P^i \right\rangle}{2\eta^i}$ and  $\overline{\xi}^{*,i}:=\frac{\overline{M}^i}{2\eta^i}$.
\end{theorem} 
\begin{proof}
Using $M^j_t=\delta M^j_t+\overline M^j_t$ and $X^j_t=\delta X^j_t+\overline X^j_t$ we have that
\begin{equation}\label{eq:diff-chi-convergence}
	\begin{split}
	\delta\chi=&~\left(\frac{1}{N}\sum_{j=1}^N\frac{\delta  M^j}{2\eta^j}+\frac{1}{N}\sum_{j=1}^N\frac{\overline  M^j}{2\eta^j}-\mathbb E\left[\frac{\overline  M^i}{2\eta^j}\right],~\frac{1}{N}\sum_{j=1}^N\mathbb E[\delta X^j]+\frac{1}{N}\sum_{j=1}^N\mathbb E[\overline X^j]-\mathbb E\left[\overline X^i\right]	\right)^\top\\
	&~+\left(-\frac{1}{N^2}\sum_{j=1}^N\frac{\left\langle B^{j,(1)}, \mathcal P^j \right\rangle}{2\eta^j} ,~0 \right)^\top\\
	=&~\left(\frac{1}{N}\sum_{j=1}^N\frac{\delta  M^j}{2\eta^j}+\frac{1}{N}\sum_{j=1}^N\frac{\overline  M^j}{2\eta^j}-\mathbb E\left[\frac{\overline  M^i}{2\eta^j}\right],~\frac{1}{N}\sum_{j=1}^N\mathbb E[\delta X^j]	\right)^\top\\
	&~+\left(-\frac{1}{N^2}\sum_{j=1}^N\frac{\left\langle B^{j,(1)}, \mathcal P^j \right\rangle}{2\eta^j} ,~0 \right)^\top.
	\end{split}
\end{equation}
In view of \eqref{ineq-Pi-general} and \eqref{estimate:S-step2}, we have that
\begin{equation}\label{estimate:P-convergence}
\begin{aligned}
\mathbb E\left[\int^T_0|\mathcal P^i_t|^2\,dt\right]&\leq\frac{1}{\widetilde{\rho}^2}\mathbb E\left[\int^T_0\left(\frac{ M^i_t}{2\eta^i_t}\right)^2\,dt\right]=\frac{1}{\widetilde{\rho}^2}\mathbb E\left[\int^T_0\left(\frac{\delta M^i_t+\overline{M}^i_t}{2\eta^i_t}\right)^2\,dt\right].
\end{aligned}
\end{equation}
and that
\begin{equation}\label{estimate:S-convergence}
	\mathbb E\left[\int_0^T|\delta\mathcal S^i_t|^2\,dt \right]\leq \frac{1}{\widehat\rho^2}\mathbb E\left[ \int_0^T |B^i_t\delta\chi_t|^2 \right].
\end{equation}
Taking \eqref{eq:diff-chi-convergence} into \eqref{eq:diff-convergence}, following the proof of Theorem \ref{fbsde-thm} and using \eqref{estimate:P-convergence} and \eqref{estimate:S-convergence}, we obtain
\begin{align*}
		&~\left(2\lambda_{\min}-\frac{\theta_0+\theta_1+\theta_2}{2}\right)\mathbb E\left[\int_0^T(\delta X^i_t)^2\,dt\right]+\left(2\eta_{\min}-\frac{\theta}{2}\right)\mathbb E\left[ \int_0^T\left( \frac{\delta M^i_t}{2\eta^i_t} \right)^2\,dt \right]\\
		\leq&~\left(\frac{\|B^{(2)}\|^2}{\theta_0}+\frac{\|B^{(1)}\|^2}{\theta}  \right)\frac{1}{2N^2}\mathbb E\left[\int_0^T|\mathcal P^i_t|^2\,dt  \right]+\left(\frac{\|A\|^2}{2\theta_1\widehat\rho^2}+\frac{1}{2\theta_2}   \right)\mathbb E\left[ \int_0^T|  B^i\delta\chi_t  |^2\,dt  \right]\\
		\leq&~\left(\frac{\|B^{(2)}\|^2}{\theta_0}+\frac{\|B^{(1)}\|^2}{\theta}  \right)\frac{1}{2N^2}\mathbb E\left[\int_0^T|\mathcal P^i_t|^2\,dt  \right]\\
		&~+(1+\theta_3)\left( \frac{\|A\|^2}{2\theta_1\widehat\rho^2}+\frac{1}{2\theta_2} \right)\|B^{(1)}\|^2\mathbb E\left[ \int_0^T\left( \frac{1}{N}\sum_{j=1}^N\frac{\delta  M^j}{2\eta^j}\right.\right.\\
		&~\left.\left.+\frac{1}{N}\sum_{j=1}^N\frac{\overline  M^j}{2\eta^j}-\mathbb E\left[\frac{\overline  M^i}{2\eta^j}\right] -\frac{1}{N^2}\sum_{j=1}^N\frac{\left\langle B^{j,(1)}, \mathcal P^j \right\rangle}{2\eta^j}  \right)^2\,dt  \right]\\
		&~+\left( 1+\frac{1}{\theta_3}  \right)\left( \frac{\|A\|^2}{2\theta_1\widehat\rho^2}+\frac{1}{2\theta_2} \right)\|B^{(2)}\|^2\frac{1}{N}\sum_{j=1}^N\mathbb E\left[ \int_0^T|\delta X^j_t|^2\,dt  \right]\\
		\leq&~\left(\frac{\|B^{(2)}\|^2}{\theta_0}+\frac{\|B^{(1)}\|^2}{\theta}  \right)\frac{1}{2N^2}  \mathbb E\left[\int_0^T|\mathcal P^i_t|^2\,dt  \right]\\
		&~+(1+\epsilon)^2(1+\theta_3)\left( \frac{\|A\|^2}{2\theta_1\widehat\rho^2}+\frac{1}{2\theta_2} \right)\|B^{(1)}\|^2\frac{1}{N}\sum_{j=1}^N\mathbb E\left[ \int_0^T\left( \frac{\delta  M^j}{2\eta^j}\right)^2\,dt\right]\\
		&~+(1+\epsilon)\left(1+\frac{1}{\epsilon}\right)(1+\theta_3)\left( \frac{\|A\|^2}{2\theta_1\widehat\rho^2}+\frac{1}{2\theta_2} \right)\|B^{(1)}\|^2\mathbb E\left[ \int_0^T\left(\frac{1}{N}\sum_{j=1}^N\frac{\overline  M^j}{2\eta^j}-\mathbb E\left[\frac{\overline  M^i}{2\eta^j}\right] \right)^2\right]\\
		&~+\left( 1+\frac{1}{\epsilon} \right)(1+\theta_3)\left( \frac{\|A\|^2}{2\theta_1\widehat\rho^2}+\frac{1}{2\theta_2} \right)\frac{\|B^{(1)}\|^4}{4\eta^2_{\min}N^2}\frac{1}{N}\sum_{j=1}^N\mathbb E\left[ \int_0^T|\mathcal P^j_t|^2\,dt  \right]\\
		&~+\left( 1+\frac{1}{\theta_3}  \right)\left( \frac{\|A\|^2}{2\theta_1\widehat\rho^2}+\frac{1}{2\theta_2} \right)\|B^{(2)}\|^2\frac{1}{N}\sum_{j=1}^N\mathbb E\left[ \int_0^T|\delta X^j_t|^2\,dt  \right]\\
		\leq&~(1+\epsilon)\left(\frac{\|B^{(2)}\|^2}{\theta_0}+\frac{\|B^{(1)}\|^2}{\theta}  \right)\frac{1}{2N^2\widetilde\rho^2}  \mathbb E\left[\int_0^T\left(\frac{\delta M^i_t}{2\eta^i_t}\right)^2\,dt  \right]\\
		&~+\left(1+\frac{1}{\epsilon}\right)\left(\frac{\|B^{(2)}\|^2}{\theta_0}+\frac{\|B^{(1)}\|^2}{\theta}  \right)\frac{1}{2N^2\widetilde\rho^2}  \mathbb E\left[\int_0^T\left(\frac{\overline M^i_t}{2\eta^i_t}\right)^2\,dt  \right]\\
		&~+(1+\epsilon)^2(1+\theta_3)\left( \frac{\|A\|^2}{2\theta_1\widehat\rho^2}+\frac{1}{2\theta_2} \right)\|B^{(1)}\|^2\frac{1}{N}\sum_{j=1}^N\mathbb E\left[ \int_0^T\left( \frac{\delta  M^j}{2\eta^j}\right)^2\,dt\right]\\
		&~+(1+\epsilon)\left(1+\frac{1}{\epsilon}\right)(1+\theta_3)\left( \frac{\|A\|^2}{2\theta_1\widehat\rho^2}+\frac{1}{2\theta_2} \right)\|B^{(1)}\|^2\mathbb E\left[ \int_0^T\left(\frac{1}{N}\sum_{j=1}^N\frac{\overline  M^j}{2\eta^j}-\mathbb E\left[\frac{\overline  M^i}{2\eta^j}\right] \right)^2\right]\\
		&~+(1+\epsilon)\left( 1+\frac{1}{\epsilon} \right)(1+\theta_3)\left( \frac{\|A\|^2}{2\theta_1\widehat\rho^2}+\frac{1}{2\theta_2} \right)\frac{\|B^{(1)}\|^4}{4\eta^2_{\min}N^2\widetilde\rho^2}\frac{1}{N}\sum_{j=1}^N\mathbb E\left[ \int_0^T\left(\frac{\delta M^j_t}{2\eta^j_t}\right)^2\,dt  \right]\\
		&~+\left( 1+\frac{1}{\epsilon} \right)^2(1+\theta_3)\left( \frac{\|A\|^2}{2\theta_1\widehat\rho^2}+\frac{1}{2\theta_2} \right)\frac{\|B^{(1)}\|^4}{4\eta^2_{\min}N^2\widetilde\rho^2}\frac{1}{N}\sum_{j=1}^N\mathbb E\left[ \int_0^T\left(\frac{\overline M^j_t}{2\eta^j_t}\right)^2\,dt  \right]\\
		&~+\left( 1+\frac{1}{\theta_3}  \right)\left( \frac{\|A\|^2}{2\theta_1\widehat\rho^2}+\frac{1}{2\theta_2} \right)\|B^{(2)}\|^2\frac{1}{N}\sum_{j=1}^N\mathbb E\left[ \int_0^T|\delta X^j_t|^2\,dt  \right].
\end{align*}
Letting $\theta=\frac{\|B^{(1)}\|}{N\widetilde\rho}$, $\epsilon$ be small enough, $N$ be large enough,
taking average and upper limit on both sides, we obtain by \eqref{ass-1}
\begin{align*}\label{c:delta-M}
&~\limsup_{N\rightarrow\infty}\frac{1}{N}\sum_{i=1}^N\mathbb E\left[\int^{T}_0 \left(\frac{\delta M^i_t}{2\eta^i_t}\right)^2\,dt\right]+\limsup_{N\rightarrow\infty}\frac{1}{N}\sum_{i=1}^N\mathbb E\left[\int^{T}_0 \left(\delta X^i_t\right)^2\,dt\right]\\
\leq&~ 
C\limsup_{N\rightarrow\infty}\frac{1}{N}\sum_{i=1}^N\mathbb E\left[\int^{T}_0\left(\frac{1}{N}\sum_{j=1}^N\frac{\overline  M^j_t}{2\eta^j_t}-\mathbb E\left[\frac{\overline  M^i_t}{2\eta^i_t}\right]\right)^2\,dt\right]\\
&~+\limsup_{N\rightarrow\infty} O\left( \frac{1}{N} \right)\frac{1}{N}\sum_{j=1}^N\mathbb E\left[ \int_0^T\left(\frac{\overline M^j_t}{2\eta^j_t}\right)^2\,dt  \right]\\
=&~0.
\end{align*}
Going back to the inequality for $\mathbb E\left[\int^{T}_0 \left(\frac{\delta M^i_t}{2\eta^i_t}\right)^2\,dt\right]$ and $\mathbb E\left[\int^{T}_0 \left(\delta X^i_t\right)^2\,dt\right]$, we have
\begin{equation*}
\mathbb E\left[\int^{T}_0 \left(\frac{\delta M^i_t}{2\eta^i_t}\right)^2\,dt\right]+\mathbb E\left[\int^{T}_0 \left(\delta X^i_t\right)^2\,dt\right]\xrightarrow{N\rightarrow\infty}0.
\end{equation*}
Furthermore,
\begin{equation*}
\begin{aligned}
\mathbb E\left[\sup_{0\leq t\leq T}|\delta X^i_t|^2\,dt\right]
\leq C
\mathbb E\left[\int_0^T\left(\delta M^{i}_t\right)^2\,dt\right]+\frac{C}{N}\mathbb E\left[\int_0^T\left(\delta M^i_s+\overline M^i_s\right)^2\,ds\right]
&\xrightarrow{N\rightarrow\infty}0.
\end{aligned}
\end{equation*}
\end{proof}

\section{Deterministic benchmark models}\label{sec:numerics}

In this section we consider the deterministic benchmark case where all model parameters except the initial portfolios are deterministic. This case is much easier to analyze and requires much weaker assumptions than the stochastic setting. For simplicity we also replace the strict liquidation constraint by a penalization $n(X^i_T)^2$ of open positions at the terminal time. This simplifies our numerical analysis; see Section \ref{sec:penalization}.

\subsection{The mean-field game}
If all model parameters except the initial positions are constant, then the stochastic integral terms drop out of the FBSDE system \eqref{MF-fbsde-2}. Taking expectations on both sides in \eqref{MF-fbsde-2} and putting  
\[
	\mathbb F:=(\mathbb E[X],\mathbb E[Y], \mathbb E[C])^\top \quad \mbox{and} \quad \mathbb B:=(\mathbb E[P], \mathbb E[Q], \mathbb E[R])^\top
\]	
we obtain that 
\begin{equation}\label{benchmark-ODE}
\left\{\begin{split}
\mathbb F'=&~\varphi_{00}\mathbb F+\varphi_{01}\mathbb B+F^0\\
\mathbb B'=&~\varphi_{10}\mathbb F+\varphi_{11}\mathbb B,\\
\mathbb F_0=&~\left(\begin{matrix}
\mathbb E[\mathcal X]\\
0\\
0
\end{matrix}\right)
,~\mathbb B_T=\left(\begin{matrix}
2n& 0& 0\\
0& 0& 0\\
0& 0& 0
\end{matrix}\right)\mathbb F_T ,
\end{split}\right.
\end{equation}
where
\[
\varphi_{00}=\left( \begin{matrix}
0& \frac{1}{2\eta}&0\\
-\alpha\gamma& -\rho-\frac{\gamma}{2\eta}&-\gamma(\beta-\alpha)\\
-\alpha& 0&-(\beta-\alpha)
\end{matrix} \right), \quad
\varphi_{01}=\left(\begin{matrix}
-\frac{1}{2\eta}&0& 0\\
\frac{\gamma}{2\eta}&0& 0\\
0& 0& 0
\end{matrix}\right), \quad 
F^0=\left( \begin{matrix}
0\\
\gamma\alpha \mathbb E[\mathcal X]\\
\alpha \mathbb E[\mathcal X]
\end{matrix} \right)
\]
and
\[
\varphi_{10}=\left( \begin{matrix}
-2\lambda& 0&0\\
0& \frac{1}{2\eta}&0\\
0& 0& 0
\end{matrix} \right), \quad
\varphi_{11}=\left(\begin{matrix}
0& 0&0\\
-\frac{1}{2\eta}& \rho&0\\
0&\gamma(\beta-\alpha)&(\beta-\alpha)
\end{matrix}\right).
\]
Making the ansatz $\mathbb B=D\mathbb F+D^0$ yields the following ODE system for $D$ and $D^0$: 
\begin{equation}\label{benchmark-riccati}
\left\{\begin{aligned}
D'=&~-D\varphi_{01} D-D\varphi_{00}+\varphi_{11} R+\varphi_{10},\qquad D_T=\left( \begin{matrix}
2n&0&0\\
0& 0&0\\
0& 0&0
\end{matrix}\right)\\
(D^0)'=&~(\varphi_{11}-R\varphi_{10})D^0-DF^0,\qquad\qquad D^0_T=(0,0,0)^\top.
\end{aligned}\right.
\end{equation}

Let $\Phi(T,t)=e^{\mathscr P(T-t)}$ be the fundamental solution to \eqref{benchmark-ODE}, where
\[
\mathscr P=\left(\begin{matrix}
\varphi_{00}& \varphi_{01}\\
\varphi_{10}&\varphi_{11}
\end{matrix}\right).
\]
From \eqref{benchmark-ODE} one has
\begin{equation*}
\begin{split}
0=&~(D_T,-I_{3\times3})\left( \begin{matrix}   \mathbb F_T\\ \mathbb B_T \end{matrix}\right)\\
=&~(D_T,-I_{3\times 3})\Phi(T,t)\left(\begin{matrix}
\mathbb F_t\\\mathbb B_t
\end{matrix}\right)+(D_T,-I_{3\times3})\int^T_t\Phi(T,s)\,ds\left(\begin{matrix}
F^0\\O_{3\times 1}
\end{matrix}\right)\\
=&~(D_T,-I_{3\times 3})\Phi(T,t)\left(\begin{matrix}
I_{3\times 3}\\O_{3\times 3}
\end{matrix}\right)	\mathbb F_t+(D_T,-I_{3\times 3})\Phi(T,t)\left(\begin{matrix}
O_{3\times 3}\\I_{3\times 3}
\end{matrix}\right)	\mathbb B_t\\
&~+(D_T,-I_{3\times3})\int^T_t\Phi(T,s)\,ds\left(\begin{matrix}
F^0\\O_{3\times 1}
\end{matrix}\right)
\end{split}
\end{equation*}
where $I_{3\times3}, O_{3\times3}$ and $O_{3\times1}$ are the $3\times3$ identity matrix and $3\times3, 3\times1$ zero matrices, respectively. If 
$(D_T,-I_{3\times3})\Phi(T,t)\left(\begin{matrix}
O_{3\times3}\\I_{3\times 3}
\end{matrix}\right)$ is invertible, which will be the case in the our simulations, a direct calculation shows that the unique solution to \eqref{benchmark-riccati} is given by
\begin{equation}\label{R}
D_t=-\left[  (D_T, -I_{3\times 3})\Phi(T,t)\left(\begin{matrix} O_{3\times 3}\\I_{3\times 3} \end{matrix}\right) \right]^{-1}( D_T, -I_{3\times 3} )\Phi(T,t)\left(\begin{matrix} I_{3\times3}\\O_{3\times3}  \end{matrix}\right)
\end{equation}
and
\begin{equation}\label{R0}
D^0_t=-\left[  (D_T, -I_{3\times 3})\Phi(T,t)\left(\begin{matrix} O_{3\times 3}\\I_{3\times 3} \end{matrix}\right) \right]^{-1}( D_T, -I_{3\times 3} )\int^T_t\Phi(T,s)\,ds\left(\begin{matrix}
F^0\\O_{3\times 1}
\end{matrix}\right).
\end{equation}

Having derived an explicit solution for the expected equilibrium portfolio process allows us to derive an explicit solution for the equilibrium portfolio process itself. It is not difficult to see  that 
\begin{equation*}
\left\{\begin{split}
(X_t-\mathbb E[X_t])'=&~-\frac{P_t-\mathbb E[P_t]}{2\eta}\\
-(P_t-\mathbb E[P_t])'=&~2\lambda(X_t-\mathbb E[X_t])\\
X_0-\mathbb E[X_0]=&~\mathcal X-\mathbb E[\mathcal X] \\
P_T-\mathbb E[P_T]=&~ 2n(X_T-\mathbb E[X_T])
\end{split}\right.
\end{equation*}
which is approximated by
\begin{equation*}
\left\{\begin{split}
(X_t-\mathbb E[X_t])'=&~-\frac{P_t-\mathbb E[P_t]}{2\eta}\\
-(P_t-\mathbb E[P_t])'=&~2\lambda(X_t-\mathbb E[X_t])\\
X_0-\mathbb E[X_0]=&~\mathcal X-\mathbb E[\mathcal X] \\
X_T-\mathbb E[X_T]=&~0.
\end{split}\right.
\end{equation*}
Making the ansatz $P-\mathbb E[P]=A(X-\mathbb E[X])$ yields
\[
	A'=\frac{A^2}{2\eta}-2\lambda,\quad A_T=\infty,
\]
or equivalently,
\[
A_t=2\sqrt{\eta\lambda}\coth\left( \sqrt{\frac{\lambda}{\eta}}(T-t) \right).
\]
Thus, we get that
\[
X_t-\mathbb E[X_t]=(\mathcal X-\mathbb E[\mathcal X])e^{-\int_0^t\frac{A_s}{2\eta}\,ds}=(\mathcal X-\mathbb E[\mathcal X])\frac{\sinh\left( \sqrt{\frac{\lambda}{\eta}}(T-t) \right)}{\sinh\left( \sqrt{\frac{\lambda}{\eta}}T \right)}
\]
and hence the optimal position approximately equals
\[
	X_t\approx \mathbb E[X_t]+(\mathcal X-\mathbb E[\mathcal X])\frac{\sinh\left( \sqrt{\frac{\lambda}{\eta}}(T-t) \right)}{\sinh\left( \sqrt{\frac{\lambda}{\eta}}T \right)}.
\]

\begin{figure}[h]\label{figure1}
	\begin{minipage}[c]{0.5\textwidth}
		\centering
		\includegraphics[height=5.7cm,width=7.5cm]{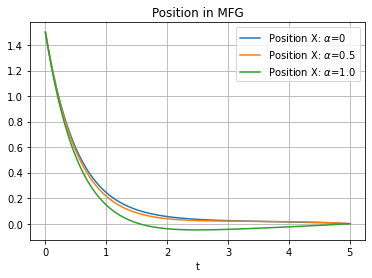}
	\end{minipage}
	\begin{minipage}[c]{0.5\textwidth}
		\centering
		\includegraphics[height=5.7cm,width=7.5cm]{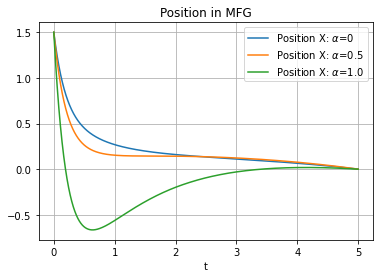}
	\end{minipage}
	\caption{Dependence of equilibrium portfolio process on the market impact parameter $\alpha$, $\gamma=0.1$(left) and $\gamma=1$(right). Other parameters are chosen as $\eta=0.1$, $\rho=0.2$, $\lambda=0.3$, $\beta=1.1$, $x=1$, $\mathbb E[\mathcal X]=1.5$ and $T=5$.}
\end{figure}

Figure \ref{figure1} displays the equilibrium portfolio processes in an MFG for varying degrees of child order flow and transient market impact. We can see from both pictures that short positions do not occur in equilibrium if the impact as measured by the quantities $\alpha$ and $\gamma$ is small. For near critical values of $\alpha$  it is optimal for the representative player to unwind his position before the terminal time, and then to take a negative position that he closes at the end of the trading period. This effect increases significantly in the impact parameter $\gamma$. The result is intuitive; the larger $\alpha$ and $\gamma$, the stronger the representative player benefits from the inertia in market order flow when closing a short position.

\subsection{Single player model}

When $N=1$ and all model parameters are deterministic constants, then our mean-field FBSDE can be rewritten as 
\begin{equation}\label{benchmark-ODE-1-player}
\left\{\begin{split}
\mathbb F'=&~ \psi_{00}\mathbb F+ \psi_{01}\mathbb B+F^0\\
\mathbb B'=&~ \psi_{10}\mathbb F+ \psi_{11}\mathbb B\\
\mathbb F_0=&~\left(\begin{matrix}
x\\
0\\
0
\end{matrix}\right)
,~\mathbb B_T=\left(\begin{matrix}
2n& 0& 0\\
0& 0& 0\\
0& 0& 0
\end{matrix}\right)\mathbb F_T ,
\end{split}\right.
\end{equation}
where 
\[
	\mathbb F=(X,Y,C) \quad \mbox{and} \quad \mathbb B=(P,Q,R)
\]
and
\[
\psi_{00}=\left( \begin{matrix}
0& \frac{1}{2\eta}&0\\
-\alpha\gamma& -\rho-\frac{\gamma}{2\eta}&-\gamma(\beta-\alpha)\\
-\alpha& 0&-(\beta-\alpha)
\end{matrix} \right), \quad
\psi_{01}=\left(\begin{matrix}
-\frac{1}{2\eta}&\frac{\gamma}{2\eta}& 0\\
\frac{\gamma}{2\eta}&-\frac{\gamma^2}{2\eta}& 0\\
0& 0& 0
\end{matrix}\right), \quad  F^0=\left( \begin{matrix}
0\\
\gamma\alpha x\\
\alpha x
\end{matrix} \right),
\]
\[
\psi_{10}=\left( \begin{matrix}
-2\lambda& 0&0\\
0& \frac{1}{2\eta}&0\\
0& 0& 0
\end{matrix} \right), \quad
\psi_{11}=\left(\begin{matrix}
0& \alpha\gamma&\alpha\\
-\frac{1}{2\eta}& \rho+\frac{\gamma}{2\eta}&0\\
0&\gamma(\beta-\alpha)&(\beta-\alpha)
\end{matrix}\right).
\]
Making again a linear ansatz $\mathbb B=\mathscr D\mathbb F+\mathscr D^0$, yields
\begin{equation}\label{benchmark-riccati-1-player}
\left\{\begin{aligned}
\mathscr D'=&~-\mathscr D\psi_{01}\mathscr D-\mathscr D\psi_{00}+\psi_{11} \mathscr D+\psi_{10},\qquad \mathscr D_T=\left( \begin{matrix}
2n&0&0\\
0& 0&0\\
0& 0&0
\end{matrix}\right)\\
(\mathscr D^0)'=&~(\psi_{11}-\mathscr D\psi_{01})\mathscr D^0-\mathscr DF^0,\qquad\qquad \mathscr D^0_T=(0,0,0)^\top
\end{aligned}\right.
\end{equation}
and the same argument as in the previous section show that the unique solution to the above ODE system is given by
\[
\mathscr D_t=-\left[  (\mathscr D_T, -I_{3\times 3})\Psi(T,t)\left(\begin{matrix} O_{3\times 3}\\I_{3\times 3} \end{matrix}\right) \right]^{-1}( \mathscr D_T, -I_{3\times 3} )\Psi(T,t)\left(\begin{matrix} I_{3\times3}\\O_{3\times3}  \end{matrix}\right)
\]
and
\[
\mathscr D^0_t=-\left[  (\mathscr D_T, -I_{3\times 3})\Psi(T,t)\left(\begin{matrix} O_{3\times 3}\\I_{3\times 3} \end{matrix}\right) \right]^{-1}( \mathscr D_T, -I_{3\times 3} )\int^T_t\Psi(T,s)\,ds\left(\begin{matrix}
F^0\\O_{3\times 1}
\end{matrix}\right)
\]
where $\Psi(T,t)=e^{\mathscr G(T-t)}$ and
\[
\mathscr G= \left(\begin{matrix}
\psi_{00}& \psi_{01}\\
\psi_{10}& \psi_{11}	
\end{matrix}\right).
\]
Note that $\int_t^T\psi(T,s)\,ds=\mathscr G^{-1}( e^{\mathscr G(T-t)}-I_{6\times 6} )$ as long as $\mathscr G$ is invertible. This is indeed the case because $\beta > \alpha$ and so 
\[
	\textrm{det}(\mathscr G)=-\rho\lambda(\beta-\alpha)^2\left( \frac{\gamma}{\eta^2}+\frac{\rho}{\eta} \right) \neq 0.
\]

In particular, the single player model with penalization can be solved explicitly. Optimal positions for various choices of model parameters are shown in Figure \ref{figure2}. The left figure shows the optimal portfolio process for various degrees of child order flow when $\gamma=1$, $\beta = 1.1$ and $\lambda = 0.3$. We see that the initial trading rate increases in $\alpha$ and that it is optimal to oversell for near-critical values of $\alpha$. The right picture shows the optimal portfolio process for different degrees of transient market impact. For very large values of $\gamma$ large fluctuations in the optimal portfolio process emerge. We emphasize that this behavior only occurs for very large values of $\gamma$.     

\begin{figure}[h]\label{figure2}
	\begin{minipage}[c]{0.5\textwidth}
		\centering
		\includegraphics[height=5.7cm,width=7.5cm]{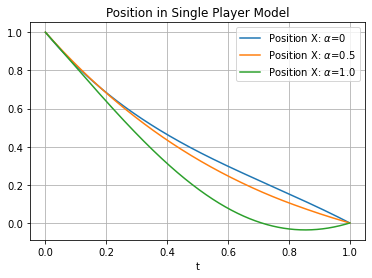}
	\end{minipage}
	\begin{minipage}[c]{0.5\textwidth}
		\centering
		\includegraphics[height=5.7cm,width=7.5cm]{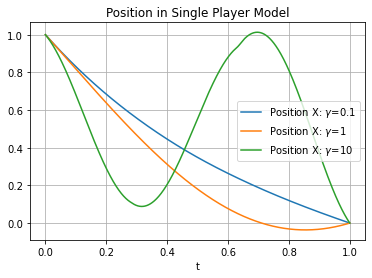}
	\end{minipage}
	\caption{Dependence of optimal portfolio process on the market impact parameters $\alpha$ and $\gamma$, $\gamma=1$(left) and $\alpha=1$(right).  Parameters are chosen as $\eta=0.1$, $\rho=0.2$, $\lambda=0.3$,  $\beta=1.1$, $x=1$ and $T=5$.}
\end{figure}

\subsection{Two player model}

If $N=2$, then our mean-field FBSDE can be rewritten as

\begin{equation}\label{benchmark-ODE-2-player}
\left\{\begin{split}
\mathbb F'=&~ \phi_{00}\mathbb F+ \phi_{01}\mathbb B+F^0\\
\mathbb B'=&~ \phi_{10}\mathbb F+ \phi_{11}\mathbb B\\
\mathbb F_0=&~\left(\begin{matrix} x^1\\0\\0\\x^2\\0\\0\end{matrix} \right),\quad    \mathbb B_T=\left(\begin{matrix}  2n&0 &0 &0&0 &0\\
0&0 & 0 & 0&0 &0\\
0&0 & 0 & 0&0 &0\\
0 & 0 & 0& 2n & 0&0\\
0&0 & 0 & 0&0 &0\\
0&0 & 0 & 0&0 &0  \end{matrix} \right)\mathbb F_T,
\end{split}\right.
\end{equation}
where 
\[
	\mathbb F=(X^{(1)},Y^{(1)},C^{(1)},X^{(2)},Y^{(2)},C^{(2)})^\top \quad \mbox{and} \quad \mathbb B=(P^{(1)},Q^{(1)},R^{(1)},P^{(2)},Q^{(2)},R^{(2)})^\top
\]	
and
\[
\phi_{00}=\left( \begin{matrix}
0& \frac{1}{2\eta_1}&0&0& 0& 0\\
-\frac{\gamma^1\alpha^1}{2}& -\rho^1-\frac{\gamma^1}{4\eta^1}&-\gamma^1(\beta^1-\alpha^1)&-\frac{\gamma^1\alpha^1}{2}& -\frac{\gamma^1}{4\eta^2}&0\\
-\frac{\alpha^1}{2}& 0&-(\beta^1-\alpha^1)&-\frac{\alpha^1}{2}& 0&0\\
0&0&0& 0& \frac{1}{2\eta_2}& 0\\
-\frac{\gamma^2\alpha^2}{2}& -\frac{\gamma^2}{4\eta^1}&0&-\frac{\gamma^2\alpha^2}{2}& -\rho^2-\frac{\gamma^2}{4\eta^2}&{-\gamma^2(\beta^2-\alpha^2)}\\
-\frac{\alpha^2}{2}& 0&0&-\frac{\alpha^2}{2}& 0&-(\beta^2-\alpha^2)
\end{matrix} \right),\]
\[ 
\phi_{01}=\left(\begin{matrix}
-\frac{1}{2\eta^1}&\frac{\gamma^1}{4\eta^1}& 0&0& 0& 0\\
\frac{\gamma^1}{4\eta^1}&-\frac{(\gamma^1)^2}{8\eta^1}& 0&\frac{\gamma^1}{4\eta^2}&-\frac{\gamma^1\gamma^2}{8\eta^2}& 0\\
0& 0& 0&0& 0& 0\\
0&0& 0&-\frac{1}{2\eta^2}&\frac{\gamma^2}{4\eta^2}&  0\\
\frac{\gamma^2}{4\eta^1}&-\frac{\gamma^1\gamma^2}{8\eta^1}& 0&\frac{\gamma^2}{4\eta^2}&-\frac{(\gamma^2)^2}{8\eta^2}& 0\\
0& 0& 0&0& 0& 0
\end{matrix}\right),
\phi_{10}=\left( \begin{matrix}
-2\lambda^1& 0&0&0& 0& 0\\
0& \frac{1}{2\eta^1}&0&0& 0& 0\\
0& 0& 0&0& 0& 0\\
0&0&0&-2\lambda^2& 0& 0\\
0& 0&0& 0&\frac{1}{2\eta^2}& 0\\
0& 0& 0&0& 0& 0
\end{matrix} \right),\]
\[ 
\phi_{11}=\left(\begin{matrix}
0& \frac{\alpha^1\gamma^1}{2}&{\frac{\alpha^1}{2}}&0& 0& 0\\
-\frac{1}{2\eta^1}& \rho^1+\frac{\gamma^1}{4\eta^1}&0&0& 0& 0\\
0&\gamma^1(\beta^1-\alpha^1)&(\beta^1-\alpha^1)&0& 0& 0\\
0& 0& 0&0& \frac{\alpha^2\gamma^2}{2}&{\frac{\alpha^2}{2}}\\
0&0& 0&-\frac{1}{2\eta^2}& \rho^2+\frac{\gamma^2}{4\eta^2}& 0\\
0&0& 0& 0&\gamma^2(\beta^2-\alpha^2)&(\beta^2-\alpha^2)
\end{matrix}\right),\]
and
\[ F^0=\left(
0,
\frac{\gamma^1\alpha^1}{2}(x^1+x^2),
\frac{\alpha^1}{2}(x^1+x^2),
0,
\frac{\gamma^2\alpha^2}{2}(x^1+x^2),
\frac{\alpha^2}{2}(x^1+x^2)
\right)^\top.
\]

Again making the ansatz $\mathbb B=\mathcal D\mathbb F+\mathcal D^0$, where
\begin{equation}\label{benchmark-riccati-1-player}
\begin{aligned}
\mathcal D'&=-\mathcal D\phi_{01}\mathcal D-\mathcal D\phi_{00}+\phi_{11} \mathcal D+\phi_{10},\qquad \mathcal D_T=\left( \begin{matrix}
2n&0 & 0 & 0& 0 & 0\\
0& 0 & 0  & 0& 0 & 0\\
0 & 0 & 0 & 0& 0 & 0\\
0 & 0  & 0& 2n& 0 & 0\\
0 & 0 & 0 & 0& 0 & 0\\
0 & 0 & 0 & 0& 0 & 0
\end{matrix}\right),\\
(\mathcal D^0)'&=(\phi_{11}-\mathcal D\phi_{01})\mathcal D^0-\mathcal DF^0,\qquad\qquad \mathcal D^0_T=\left( 
0,
0,
0,
0,
0,
0 \right)^\top
\end{aligned}
\end{equation}
the same arguments as in the mean-field case yield the unique solution 
\[
\mathcal D_t=-\left[  (\mathcal D_T, -I_{6\times6})\Phi(T,t)\left(\begin{matrix} O_{6\times6}\\I_{6\times6} \end{matrix}\right) \right]^{-1}( \mathcal D_T, -I_{6\times6} )\Phi(T,t)\left(\begin{matrix} I_{6\times6}\\O_{6\times6}  \end{matrix}\right),
\]
and
\[
\mathcal D^0_t=-\left[  (\mathcal D_T, -I_{6\times6})\Phi(T,t)\left(\begin{matrix} O_{6\times6}\\I_{6\times6} \end{matrix}\right) \right]^{-1}( \mathcal D_T, -I_{6\times 6} )\int^T_t\Phi(T,s)\,ds\left(\begin{matrix}
F^0\\O_{6\times 1}
\end{matrix}\right)
\]
where $\Phi(T,t)=e^{\mathcal G(T-t)}$ and
\[
\mathcal G= \left(\begin{matrix}
\phi_{00}& \phi_{01}\\
\phi_{10}& \phi_{11}	
\end{matrix}\right).
\]

There is no explicit expression for the integral since $\textrm{det}(\mathcal G)\equiv 0$. Figure \ref{figure3} shows equilibrium positions in a two player model with different degrees of transient market impact. In both cases, Player 2 benefits from the presence of Player 1; there is a beneficial round-trip for this player in equilibrium. As expected the round-trip is stronger (more convex) for larger degrees of transient impact.   

\begin{figure}[h]\label{figure3}
	\begin{minipage}[c]{0.5\textwidth}
		\centering
		\includegraphics[height=5.7cm,width=7.5cm]{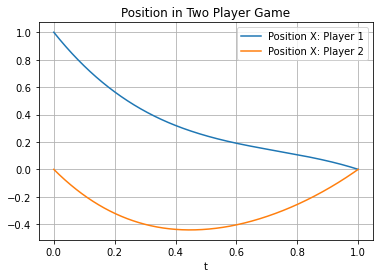}
	\end{minipage}
	\begin{minipage}[c]{0.5\textwidth}
		\centering
		\includegraphics[height=5.7cm,width=7.5cm]{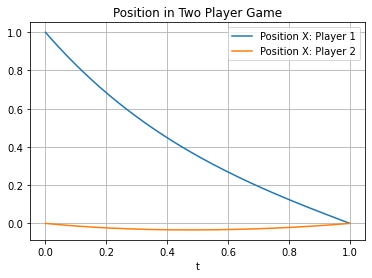}
	\end{minipage}
	\caption{Equilibrium portfolio process in the two player game under different parameters $\gamma$, $\gamma=1$(left) and $\gamma=0.1$(right). Other parameters are chosen as $\eta_1=\eta_2=0.1$, $\rho_1=\rho_2=0.2$, $\lambda_1=\lambda_2=0.3$, $\alpha_1=\alpha_2=1$,  $\beta_1=\beta_2=1.1$, $x_1=1, x_2=0$, and $T=5$.}
\end{figure}

\bibliographystyle{siam}
{
	\bibliography{bib_FHX}
}
\end{document}